\documentclass[11pt,a4paper]{amsart}
\usepackage[T1]{fontenc}
\usepackage[utf8]{inputenc}
\usepackage{amsmath,amssymb, amsbsy}
\usepackage{subfigure}
\usepackage{graphpap,latexsym,epsf}
\usepackage{color,psfrag}
\usepackage[dvips]{graphicx}
\usepackage[hidelinks]{hyperref}  

\usepackage[authoryear]{natbib}

\usepackage{enumerate}
\usepackage{bbm}
\usepackage{relsize}
\usepackage{booktabs}
\renewcommand{\leq}{\leqslant}

\begin{document}
\newcommand{\dyle}{\displaystyle}
\newcommand{\R}{{\mathbb{R}}}
\newcommand{\Hi}{{\mathbb H}}
\newcommand{\Ss}{{\mathbb S}}
\newcommand{\N}{{\mathbb N}}
\newcommand{\Rn}{{\mathbb{R}^n}}
\newcommand{\F}{{\mathcal F}}
\newcommand{\ieq}{\begin{equation}}
\newcommand{\eeq}{\end{equation}}
\newcommand{\ieqa}{\begin{eqnarray}}
\newcommand{\eeqa}{\end{eqnarray}}
\newcommand{\ieqas}{\begin{eqnarray*}}
\newcommand{\eeqas}{\end{eqnarray*}}
\newcommand{\f}{\hat{f}}
\newcommand{\Bo}{\put(260,0){\rule{2mm}{2mm}}\\}
\newcommand{\1}{\mathlarger{\mathlarger{\mathbbm{1}}}}


\theoremstyle{plain}
\newtheorem{theorem}{Theorem} [section]
\newtheorem{corollary}[theorem]{Corollary}
\newtheorem{lemma}[theorem]{Lemma}
\newtheorem{proposition}[theorem]{Proposition}
\def\neweq#1{\begin{equation}\label{#1}}
\def\endeq{\end{equation}}
\def\eq#1{(\ref{#1})}


\theoremstyle{definition}
\newtheorem{definition}[theorem]{Definition}
\newtheorem{remark}[theorem]{Remark}
\numberwithin{figure}{section}

\title[Beneath the kinetic interpretation of noise]{Beneath the kinetic interpretation of noise}

\author[C. Escudero]{Carlos Escudero}
\author[H. Rojas]{Helder Rojas}
\address{}
\email{}


\date{\today}

\begin{abstract}
Diffusion theory establishes a fundamental connection between stochastic differential equations and partial differential equations. The solution of a partial differential equation known as the Fokker-Planck equation describes the probability density of the stochastic process that solves a corresponding stochastic differential equation. The kinetic interpretation of noise refers to a prospective notion of stochastic integration that would connect a stochastic differential equation with a Fokker-Planck equation consistent with the Fick law of diffusion, without introducing correction terms in the drift. This work is devoted to identifying the precise conditions under which such a correspondence can occur. One of these conditions is a structural constraint on the diffusion tensor, which severely restricts its possible form and thereby renders the kinetic interpretation of noise a non-generic situation. This point is illustrated through a series of examples. Furthermore, the analysis raises additional questions, including the possibility of defining a stochastic integral inspired by numerical algorithms, the behavior of stochastic transport equations in heterogeneous media, and the development of alternative models for anomalous diffusion. All these topics are addressed using stochastic analytical tools similar to those employed to study the main problem: the existence of the kinetic interpretation of noise.
\end{abstract}
\maketitle

\section{Introduction}\label{intro}

Diffusion is usually considered as a fundamental transport process in many sciences, including physics, chemistry, biology, and even finance. Since the seminal work of Adolf Fick~\cite{fick1855}, diffusive transport has frequently been understood macroscopically through the constitutive law
\begin{equation}\nonumber
  \mathbf{J} = -\mathbf{D}\nabla u,
\end{equation}
which postulates that the flux $\mathbf{J}$ of a conserved scalar quantity $u$ (such as mass density) is proportional to the gradient of that quantity, where $\mathbf{D}$ denotes the diffusion tensor. This law is normally combined with the continuity equation to yield the well-known family of \emph{convection-diffusion equations}. Today, this family constitutes one of the most common classes of mathematical models for transport phenomena. Despite its mathematical simplicity, the Fick law embodies a deep physical idea: it gives a macroscopic representation of the irreversible relaxation towards equilibrium. Half a century later, Einstein showed how this law emerges from the underlying stochastic dynamics of Brownian motion~\cite{einstein1905}. That connection was further underpinned by subsequent work on the \emph{fluctuation-dissipation relation}~\cite{callen1951,kubo1966,zwanzig2001}. But quite remarkably, the pioneering work that connected microscopic fluctuations to macroscopic diffusion was that of Bachelier in the field of mathematical finance~\cite{bachelier}. This just highlights the interdisciplinary importance of diffusive transport.

Microscopically, the trajectories of diffusive particles (such as point masses) are commonly represented by stochastic differential equations (SDEs):
\begin{equation}\label{1stsde}
  \mathrm{d}X_t = b(X_t)\,\mathrm{d}t + \sigma(X_t)\,\mathrm{d}W_t,
\end{equation}
where the noise amplitude $\sigma$ is related to the diffusion tensor via $\mathbf{D} = \sigma\sigma^{\mathsf{T}}$.
For systems with \emph{state-dependent diffusivity}, i.e.\ $\mathbf{D}=\mathbf{D}(x)$, that equation becomes highly nontrivial and needs to be understood within the theory of stochastic integration by Itô~\cite{ito1,ito2}. Although this theory is enough to connect microscopic trajectories to the macroscopic appearance of diffusion~\cite{kuo,oksendal}, this has not stopped a broad discussion in the physical literature about the convenience of substituting the Itô integral by some newer notion of stochastic integration, prominently that of Stratonovich~\cite{stratonovich}. For some recent discussions on the topic, we refer the reader to~\cite{bhattacharyay2020generalization,bhattacharyay2025active,ce,dhawan2024ito,cescudero,cescudero2,em}. This \textit{interpretation of noise dilemma}, consequently, becomes relevant when it comes to diffusive transport in nonuniform and anisotropic media, where the diffusion tensor varies spatially. This happens, for example, across diverse geophysical settings. These include sediment and bedload transport, solute migration in heterogeneous and variably saturated aquifers, vertical mixing and buoyancy effects in the ocean, and other geomorphological and hydrogeological systems~\cite{ancey2015sads,henri2022rwpt,nordam2023gmd,perez2019reactive,perezillanes2024modpathrw,spivakovskaya2007lagrangian}.

Not surprisingly, more interpretations of noise have been explored throughout the years. Indeed, the dependence of a stochastic integral, thought of in a Riemann-sum-like fashion, on the evaluation point of the integrand has been a subject of study for decades. Herein, we depart from reference~\cite{escudero2023ito}, where we carried out a comparison among three different interpretations of noise —namely, the Itô, Stratonovich, and Hänggi–Klimontovich (HK) calculi— within a stochastic differential equation. These interpretations are the most natural from a symmetry viewpoint, as they correspond to evaluating an integrand, respectively, at the left endpoint, middle point, and right endpoint of each subinterval in the partition generated by the Riemann sum approximation. That work showed how each convention alters both the paths of the process and the evolution of their probability densities from a rigorous mathematical perspective, highlighting the mutual comparison between them. In the present work, we depart from there but focus on the connection of convection-diffusion equations with stochastic processes. Precisely, we consider equations of the form
\begin{equation}\label{CDE}
\left\{
\begin{array}{ll}
   \partial_t u= - \nabla \cdot(b({x})\,u)+\frac{1}{2}\nabla \cdot(\mathbf{D}({x})\nabla u), & {x} \in \mathbb{R}^d , \, t > 0,\\
u(x, 0) = u_0(x), & x \in \mathbb{R}^d, \\
\end{array}
\right.
\end{equation}
posed in an arbitrary dimension $d$, and where the diffusion tensor $\mathbf{D}({x})$  and the velocity field $b(x)$ are spatially dependent. As already mentioned, such equations appear frequently in physics, where $u = u(x,t)$ commonly represents the density of a diffusive material at point $x$ and time $t$. They describe the spatio-temporal dispersal of particles subjected simultaneously to a convective field and a heterogeneous diffusion that is compatible with Fick law (modeled, respectively, by the first and second terms in the right-hand side of that equation). One of the main goals of this paper is to investigate the precise connection between this type of convection-diffusion equations and associated multidimensional SDEs of the form
\[
dX_t = b(X_t) \,\mathrm{d}t + \sigma(X_t) \bullet \mathrm{d}W_t,
\]
where $\sigma$ satisfies $\mathbf{D} = \sigma \sigma^\top$ and $\bullet$ denotes the Hänggi–Klimontovich integral (as mentioned before, the variant of the Stratonovich integral in which the evaluation point of the integrand is taken to be the right endpoint of each subinterval in the partition). That relation —namely, the former is the Fokker-Planck equation of the latter— is classical in one dimension~\cite{escudero2023ito}, but known to fail in arbitrary dimensions~\cite{hutter1998}. Herein, we will show that under the structural condition $\nabla \cdot \mathbf{D} = 2 \sigma \nabla \cdot \sigma^\top$, the probability density of the process $X_t$ does obey the partial differential equation for $u$, provided enough regularity is granted. However, this does not hold when such a condition is not met, which is the generic case in any dimension equal to or greater than two. Reference~\cite{hutter1998} raises the question of the existence of a discretization scheme for SDEs that directly yielded equation~\eqref{CDE} as its Fokker-Planck equation. Of course, the solution to that equation corresponds to the probability density of a process that can be shown to obey a family of Itô SDEs, all of which present explicit correction terms in the drift. However, the existence of a notion of stochastic integration —the \textit{kinetic interpretation of noise} from~\cite{hutter1998}— that is able to provide the desired connection with neither the correction terms in the drift nor the fulfillment of the structural condition seems to be not possible. Showing this, as well as complementary questions, is the goal of the present work.

The outline of this paper is as follows. In Section~\ref{sec:cdequations} we make precise the connection between convection-diffusion equations and Itô SDEs we will employ thereafter. In Section~\ref{sec:multidimensional} we introduce the multidimensional version of the HK integral. In Section~\ref{hkcalculus} we explore the relation between HK and Itô SDEs. In Section~\ref{sec:cdehkc} we establish the connection between convection-diffusion equations and HK SDEs in the absence of drift correction, which relies on a key structure of the divergence of the diffusion tensor. In that section, we also explore examples where the condition on $\nabla \cdot \mathbf{D}$ either holds or fails, thereby illustrating the impact of the properties of the diffusion tensor on the coupling between HK SDEs and convection-diffusion equations. In Section~\ref{sec:beyond} we address a number of complementary questions: the ill-posedness of an alternative discretization scheme for the kinetic interpretation of noise, which was motivated by a numerical algorithm; a complete derivation of another discretization scheme motivated by a different numerical algorithm, to illustrate how such a program can be carried out; the application of that scheme on a physically motivated problem, which moreover illustrates the difficulties that arise when one of our main assumptions is removed, namely, uniform ellipticity; and, lastly, the assessment of another physical model that replaces degenerate ellipticity by a more tractable analytical structure. Finally, in Section~\ref{sec:conclusions}, we draw our main conclusions.

\subsection*{Basic notations and setting}

Let $d$ and $m$ be positive integers. We consider $\mathbb{R}^m$ endowed with the Euclidean norm $|\cdot|$, and let $\mathcal{M}_{d\times m}(\mathbb{R})$ be the set of all $d\times m$ matrices with real coefficients, endowed with the Frobenius (Hilbert-Schmidt)  norm, defined as
$$
||B||_{F}:=\bigg(\sum\limits_{i=1}^{d}\sum\limits_{j=1}^{m} b^2_{ij}\bigg)^{1/2}.
$$
The spectral norm, defined as
$$
||B||_{S}:=\max_{\{x\in\mathbb{R}^m:\,|x|=1\}}|Bx|,$$
where $B=(b_{ij})\in\mathcal{M}_{d\times m}(\mathbb{R})$, will also be considered for matrices. Let $\mathbb{O}(d)$ be the set of all $d\times d$ orthogonal matrices, i.e. $\mathbb{O}(d):=\{R\in\mathcal{M}_{d\times d}(\mathbb{R})\,:\,R^{\top}R=\mathbb{I}_d\}$. Furthermore, for some positive real number $T$, let $\Psi:\mathbb{R}^m\times[0,T]\longrightarrow\mathcal{M}_{d\times m}(\mathbb{R})$ be a mapping for which the $m$-dimensional vector $\Psi_{i*}$ denotes the $i$'th row of the $d\times m$ matrix $\Psi=(\psi_{ij})$. Analogously, for $1\leq k \leq m$, the $d$-dimensional vector $\big(\partial_{k}\Psi\big)_{* l}$ denotes the $l$'th column of the $d\times m$ matrix $\partial_{k}\Psi=\Big({\partial_{k}\psi_{ij}}\Big)$. If $\sigma=(\sigma_{ij}) \in \mathcal{M}_{d\times m}(\mathbb{R})$, we write $\nabla \,\sigma=\big({\partial_{k}\sigma_{ij}}\big)_{ijk}$ and
$$
\nabla \cdot\sigma=\bigg(\sum\limits_{j=1}^{m}{\partial_{j} \sigma_{ij}}\bigg)
$$
to denote the gradient and the divergence of $\sigma$, respectively. If $A=(a_{ijk})$ is a third-order tensor as, for example, in the case of $\nabla\,\sigma$, we write
$$
A:B=\bigg(\sum\limits_{j=1}^d\sum\limits_{k=1}^m a_{ijk}b_{jk}\bigg)
$$
to denote the double vertical dot product. We use the abbreviation $||\cdot||_{p}:=||\cdot||_{L^p(\mathbb{R}^d)}$, $1 \le p \le \infty$, and $C_c^{\infty}(\mathbb{R}^d)$ denotes the class of smooth functions with compact support. In addition, for $U\subset\mathbb{R}^d$ open and $k\in\mathbb{N}_0$, we write $C_b^k(U)$ for the class of scalar functions $f:U\to\mathbb{R}$ such that, for every multi-index $\alpha=(\alpha_1,\dots,\alpha_d)\in\mathbb{N}_0^d$ with $|\alpha|:=\sum_{i=1}^d \alpha_i\le k$, the partial derivative $D^\alpha f:=\partial_{1}^{\alpha_1}\cdots \partial_{d}^{\alpha_d}f$ exists, is continuous, and satisfies 
\[
\sup_{x\in U}|D^\alpha f(x)|<\infty,
\]
where $D^0 f=f$, $D^{e_i}f=\partial_{i}f$ and $e_i$ is a member of the canonical basis of $\mathbb{R}^d$. When $U=\mathbb{R}^d$, we write $C_b^k(\mathbb{R}^d)$. For matrix fields $A:U\to\mathcal{M}_{d\times d}(\mathbb{R})$ we use the componentwise convention: $A\in C_b^k(U)$ iff each entry $A_{ij}\in C_b^k(U)$. From now on, we will work on a filtered probability space $(\Omega, \mathbb{F}, \mathcal{F}_t, \mathbb{P})$ completed with the $\mathbb{P}$-null sets and where a $m$-dimensional Brownian motion $W_t$ is defined. Moreover, we assume that the right-continuous filtration $\mathcal{F}_t \supseteq \mathcal{F}_t^W$, where $\mathcal{F}_t^W$ is the natural filtration of $W_t$.

\section{Convection-Diffusion Equations and Itô Calculus}\label{sec:cdequations}

The connection between partial differential equations of convection-diffusion type and Itô stochastic differential equations has been established since long ago. In this section, we start specifying some known facts, and then move to introduce the precise statements we will need in the remainder of this work.

Equation~\eqref{CDE} can be rewritten as
\begin{equation}\label{opeEquation}
\left\{
\begin{array}{ll}
   \partial_t u (x,t)= \mathcal{L} u(x,t), & {x} \in \mathbb{R}^d, \, t > 0,\\
u(x, 0) = u_0(x), & x \in \mathbb{R}^d, \\
\end{array}
\right.
\end{equation}
where the convection-diffusion operator $\mathcal{L}$ is
\begin{equation}\label{CDO}
    \mathcal{L} u(x):=-\sum\limits_{i=1}^{d} \partial_{i}(b_{i}(x)u(x))+\frac{1}{2}\sum\limits_{i,j=1}^{d} \partial_{i}(\mathbf{D}_{ij}(x)\partial_{j}u(x)),
\end{equation}
where $b=(b_{1}\cdots, b_{d}):\mathbb{R}^d \longrightarrow \mathbb{R}^d$ and $\mathbf{D}:\mathbb{R}^d \longrightarrow \mathcal{M}_{d\times d}(\mathbb{R})$ are (for the time being) continuously differentiable functions. From now on, unless explicitly specified, we assume $\mathcal{L}$ to be a uniformly elliptic operator in divergence form, so the resulting equation is parabolic~\cite{evanspde}. Moreover, under suitable assumptions on the parameters, this equation possesses a solution that is unique and classical~\cite{evanspde}. Note that the uniform ellipticity condition is standard within mathematical analysis. It also holds in many physical cases, but it cannot be considered universal within that realm. Sometimes, a degenerate elliptic operator is needed, as illustrated in Subsection~\ref{sec:heterogeneous}, where some of the difficulties that might arise in such a case are shown.

On the other hand, consider the Itô stochastic differential equation in $\mathbb{R}^d$:
\begin{equation}\label{SDE-ITO}
\mathrm{d}X_t= \Big[b(X_t)+\frac{1}{2}\nabla\cdot\mathbf{D}(X_t)\Big]\,\mathrm{d}t+\sigma(X_t)\,\mathrm{d}W_t, \qquad X_0=Z,
\end{equation}
with $W_t$ a $m$-dimensional Brownian motion and $Z$ a $\mathcal{F}_0$-measurable random variable in $\mathbb{R}^d$, where
$\{\mathcal{F}_t\}_{t \ge 0}$ is the initially enlarged filtration $\mathcal{F}_t=\mathcal{F}_t^W \vee \sigma(Z)$; moreover, assume that the law of $Z$ is absolutely continuous with respect to Lebesgue measure. For the time being, we take $\sigma:\mathbb{R}^d \longrightarrow \mathcal{M}_{d\times m}(\mathbb{R})$ to be any function such that $\mathbf{D}=\sigma \sigma^{\top}$. From the diffusion theory for Itô SDEs it is known that, under adequate assumptions on the parameters and initial condition, the solution $X_t$ of~\eqref{SDE-ITO} is a Markov process with infinitesimal generator $ \mathcal{L}^*$ defined by
\begin{equation}\label{}  \mathcal{L}^*\phi(x):=\sum\limits_{i=1}^{d}\Big[b(x)+\frac{1}{2}\nabla\cdot\mathbf{D}(x)\Big]_{i} \partial_{i}\phi(x)+\frac{1}{2}\sum\limits_{i,j=1}^{d}\mathbf{D}_{ij}(x) \partial_{i}\partial_{j}\phi(x), \qquad \textrm{for all}\qquad \phi\in C_c^{\infty}(\mathbb{R}^d),
\end{equation}
where $\mathcal{L}^*$ is the adjoint operator of $\mathcal{L}$~\cite{kuo, oksendal,protter2005stochastic}.

If we depart from equation~\eqref{opeEquation}, which we consider to be our datum, and we look for paths of a stochastic process that underlies it, in the sense that it comes as a strong solution to a precise SDE, we have already encountered our first trouble. Namely, that the correspondence of the partial differential equation (PDE) to the SDE is not unique. This happens in the simplest possible case: one-dimensional Brownian motion. In such a case, the heat equation (subjected to a unit Dirac mass initial condition):
$$
\partial_t u (x,t)= \Delta u(x,t)
$$
corresponds to two SDEs:
$$
\mathrm{d}X_t= +\mathrm{d}W_t \qquad \text{and} \qquad \mathrm{d}X_t= -\mathrm{d}W_t,
$$
from where we find two solutions $X_t= \pm W_t$ (up to an additive constant that equals the initial condition). Of course, both solutions are equal in law, but they differ pathwise. As herein we are mostly interested in strong solutions to SDEs, we need to specify which solution we select. In order to do so, first of all we need to take $\sigma:\mathbb{R}^d \longrightarrow \mathcal{M}_{d \times d}(\mathbb{R})$, that is $m=d$. Although $m>d$ is perfectly possible, this only increases the number of solutions, and at the price of introducing more noises than physical dimensions, which we consider unnatural in the absence of a modeling justification. Even in such a case, we have to extract the square root of a given $\mathbf{D}$, i.e. we look for a matrix $\sigma$ such that $\mathbf{D}=\sigma \sigma^{\top}$. It is well known that there exist infinitely many solutions to this problem (except in the scalar case, which admits just two), but one can fix uniqueness by selecting the unique symmetric positive definite square matrix $\sigma$ (more solutions are generated by multiplication of this by an orthogonal matrix). From now on we will assume this, the so-called \emph{principal square root} of $\sigma$, for the sake of symmetry and due to the fact that, although different choices lead to solutions that differ pathwise, all share the same law. But note that different choices are very well possible, like for example the matrix given by the Cholesky decomposition of $\mathbf{D}$, which might be even more relevant for some purposes~\cite{golub2013matrix}.

\begin{remark}
The existence of such a $\sigma$ is guaranteed for any symmetric and positive definite matrix $\mathbf{D}$, which in turn is guaranteed by the assumption of uniform ellipticity. The symmetry of the diffusion tensor reflects the physical Onsager reciprocity and the mathematical fact that any asymmetric contribution can be absorbed by the drift term. Note, on the other hand, that if $\mathbf{D}$ is not symmetric, then there is no $\sigma$ such that $\mathbf{D}=\sigma \sigma^{\top}$.
\end{remark}

\begin{remark}
It is interesting to note that the existence of such a square root of $\mathbf{D}$ is guaranteed provided it is just positive semidefinite. This could correspond, however, to a degenerate elliptic operator and therefore will not be treated here, since that presents other complications (see Subsection~\ref{sec:heterogeneous}). Nevertheless, in such a case it is also possible to fix uniqueness of the square root by selecting the only $\sigma$ that is a symmetric positive semidefinite square matrix.
\end{remark}

Now that we have made precise our selection of $\sigma$, we move on to establish the precise assumptions we will employ to build the relationship we seek between~\eqref{opeEquation} and~\eqref{SDE-ITO}. Namely, we assume the following hypotheses:
\begin{itemize}
    \item[(H1)] (Regularity of the parameters)
    $b:\mathbb{R}^d \longrightarrow \mathbb{R}^d$ and  $\mathbf{D}:\mathbb{R}^d \longrightarrow \mathcal{M}_{d\times d}(\mathbb{R})$ are $C^4$ and $C^5$ functions respectively, and there exists a constant $K > 0$ such that for all $x \in \mathbb{R}^d$, and $i, j, k, l = 1, \ldots, d$,
    $$
    |\partial_k b_i(x)|+ |\partial_k \mathbf{D}_{ij}(x)| + |\partial_k \partial_l \mathbf{D}_{ij}(x)| \leq K.
    $$
    \item[(H2)] (Uniform ellipticity of the diffusion operator) There exists $\alpha > 0$ such that for all $x, \xi \in \mathbb{R}^d$,
    $$
    \xi^{\top} \mathbf{D}(x) \xi \geq \alpha |\xi|^2.
    $$
    \item[(H3)] (Compatibility of the initial conditions) $Z$ and $\{W_t\}_{t \ge 0}$ are independent, and $u_0 \in L^1(\mathbb{R}^d)$ is the probability density function (PDF) of $Z$, therefore $u_0\geq 0$ and $||u_0||_1=1$.
\end{itemize}

\begin{remark}\label{globalcont}
Note that (H1) implies that the functions $b$ and $\mathbf{D}$ are globally Lipschitz continuous, but they could be unbounded. 
\end{remark}

The following results establish a connection between PDE~\eqref{opeEquation} and SDE~\eqref{SDE-ITO} on a finite time interval $[0,T]$ with $T>0$ arbitrary.

\begin{lemma}[Standard Fokker-Planck form]
\label{lem:sfpf}
Under (H1), the initial value problems~\eqref{opeEquation} and
\begin{equation}\label{CDE0}
\left\{
\begin{array}{ll}
   \partial_t u (x,t)= \mathcal{L}_{0} \, u(x,t), & {x} \in \mathbb{R}^d, \, t > 0,\\
u(x, 0) = u_0(x), & x \in \mathbb{R}^d, \\
\end{array}
\right.
\end{equation}
where
\begin{equation}\label{FPO}
\mathcal{L}_0 \, u = -\sum_{i=1}^d \partial_i (\tilde{b}_i u) + \frac{1}{2} \sum_{i,j=1}^d \partial_i \partial_j (D_{ij} u)
\end{equation}
and $\tilde{b}_i := b_i + \frac{1}{2} (\nabla \cdot \mathbf{D})_i = b
_i + \frac{1}{2} \sum_{j=1}^d \partial_j \mathbf{D}_{ij}$, coincide on classical solutions.
\end{lemma}

\begin{proof}
From~\eqref{CDO}, we have the following string of equalities
\begin{align*}
\partial_t u &= - \nabla \cdot (b u) + \frac{1}{2} \nabla \cdot (\mathbf{D} \nabla u)\\
&= - \nabla \cdot (b u) 
+ \frac{1}{2} \nabla \cdot \big( \nabla \cdot (\mathbf{D}u) - u \nabla \cdot \mathbf{D} \big) \\
&= - \nabla \cdot \bigg(\Big[b+\frac{1}{2} \nabla \cdot \mathbf{D}\Big]u\bigg) 
+ \frac{1}{2} \nabla \cdot \Big( \nabla \cdot (\mathbf{D}u)\Big) \\
&= - \sum_{i=1}^{d} \partial_{i} \left( \left[ b + \frac{1}{2} \nabla \cdot \mathbf{D} \right]_i u \right) 
+ \frac{1}{2} \sum_{i,j=1}^{d} \partial_{i} \partial_{j} \left( \mathbf{D}_{ij} u \right),
\end{align*}
provided $u$ is a classical solution of either~\eqref{opeEquation} or~\eqref{CDE0}, i.e. \(u \in C^1((0, T]; C(\mathbb{R}^d)) \cap C((0, T]; C^2(\mathbb{R}^d)) \).
\end{proof}

\begin{lemma}\label{regularity}
Let $\mathbf{D}:\mathbb{R}^d \longrightarrow \mathcal{M}_{d\times d}(\mathbb{R})$
be a symmetric matrix that satisfies the following two conditions:

\begin{enumerate}
    \item (Uniform ellipticity) There exists $\alpha > 0$ such that for all $x, \xi \in \mathbb{R}^d$,
    $$
    \xi^{\top} \mathbf{D}(x) \xi \geq \alpha |\xi|^2.
    $$
    
    \item (First-order regularity) There exists a
constant $M_k$ such that for all \( k \in \{1, \ldots, d\} \) it holds that
    \[
    \sup_{x \in \mathbb{R}^d} \max_{1 \leqslant i,j \leq d} 
    \left| \partial_{k} \mathbf{D}_{ij}(x) \right| \leq M_k.
    \]
\end{enumerate}

Let $\sigma:\mathbb{R}^d \longrightarrow \mathcal{M}_{d\times d}(\mathbb{R})$ be the principal root of \(\mathbf{D}\), i.e., the unique symmetric positive definite matrix satisfying $\mathbf{D}=\sigma \sigma^{\top}$. Then, for all \( k \in \{1, \ldots, d\} \) and \( i,j \in \{1, \ldots, d\} \), it holds that
\[
\sup_{x \in \mathbb{R}^d} 
\left| \partial_{k} \sigma_{ij}(x) \right|
\leqslant \frac{d^2}{2\sqrt{\alpha}} M_k.
\]
\end{lemma}

\begin{proof}
We depart from the identity $\mathbf{D}=\sigma \sigma^{\top}$. Differentiating with respect to the $k$-th coordinate of $x$, we obtain
\begin{equation}\label{sylvester}
    \sigma \sigma^{\prime} + \sigma^{\prime}\sigma=\mathbf{D}^{\prime},
\end{equation}
where $\sigma^{\prime}=\partial_k \sigma(x)$ and $\mathbf{D}^{\prime}=\partial_k \mathbf{D}(x)$. Since $\mathbf{D}$ is positive definite, then $\sigma$ is also positive definite. Therefore, since the $\sigma$ and $-\sigma$ matrices do not have common eigenvalues, using the spectral criterion for Sylvester equations, the linear Lyapunov operator 
\begin{equation*}
  \mathcal{L}_{\sigma}: \mathcal{M}_{d\times d} (\mathbb{R})\longrightarrow \mathcal{M}_{d\times d} (\mathbb{R}), \qquad \mathcal{L}_{\sigma}(X)= \sigma X + X \sigma,  
\end{equation*}
is invertible, which ensures that the Sylvester equation ~\eqref{sylvester} admits a unique solution $\sigma^{\prime}=\mathcal{L}^{-1}_{\sigma}(\mathbf{D}^{\prime})$, see \cite{bhatia2009positive, horn1994topics}. Subsequently, to obtain an explicit form of $\sigma^{\prime}$, we need to obtain the spectral decompositions of $\mathbf{D}$ and $\sigma$. It is well known that for every $x\in\mathbb{R}^d$ there exists an orthogonal $P\in \mathcal{M}_{d\times d}(\mathbb{R})$ such that
\begin{equation*}
   \mathbf{D} = P \Lambda P^{\top} \quad\text{and}\quad \sigma = P \Lambda^{\frac{1}{2}}\, P^{\top}, 
\end{equation*}
where $\Lambda=\text{diag}(\lambda, \dots, \lambda_d)$ and $\Lambda^{\frac{1}{2}}=\text{diag}(\sqrt{\lambda_1}, \dots, \sqrt{\lambda_d})$. Defining $H =(h_{ij}):= P^{\top} \mathbf{D}^{\prime} P$ and $R =(r_{ij}):=P^{\top} \sigma^{\prime}  P$, and multiplying equation ~\eqref{sylvester} from the left by $P^{\top}$ and from the right by \( P \), yields
\begin{equation}\label{syvestr2}
\Lambda^{\frac{1}{2}}\,R + R \,\Lambda^{\frac{1}{2}} = H.   
\end{equation}
In coordinates, equation ~\eqref{syvestr2} becomes
\[
\sqrt{\lambda_i} \, r_{ij} + r_{ij} \sqrt{\lambda_j} = h_{ij},
\]
where \( 1 \leq i, j \leq d \). 
Given that \( \sqrt{\lambda_i} + \sqrt{\lambda_j} > 0 \), we obtain
\begin{equation*}
 r_{ij} = \frac{h_{ij}}{\sqrt{\lambda_i} + \sqrt{\lambda_j}}.   \end{equation*}
Since by hypothesis $\mathbf{D}$ enjoys uniform ellipticity, that is, its smallest eigenvalue $\lambda_{\min}>\alpha\geq0$, the precedent equality implies
\begin{equation}\label{Bound-1}
  |r_{ij}| \leq \frac{|h_{ij}|}{2\sqrt{\alpha}}.   
\end{equation}
Subsequently, from the fact that $\sigma^{\prime} =PRP^{\top}$, we obtain the equality
\begin{equation*}
    \partial_k \sigma_{ij}(x)=e_{i}^{\top}PRP^{\top}e_{j}=u^{\top}Rv,
\end{equation*}
where $e_i, e_j$ are members of the canonical basis of $\mathbb{R}^d$ and $u, v\in\mathbb{R}^d$ are orthonormal vectors. Therefore, for all $x\in\mathbb{R}^d$, we have that
\begin{equation}\label{Bound-2}
    |\partial_k \sigma_{ij}(x)|=|u^{\top}Rv|\leq||R||_{S}\leq||R||_{F}\leq d\,\Big(\max_{a,b} |r_{ab}|\Big),
\end{equation}
where \( 1 \leq a, b \leq d \). Analogously, since by definition $H = P^{\top} \mathbf{D}^{\prime} P$ for all $x\in\mathbb{R}^d$, we obtain
\begin{equation}\label{Bound-3}
    |h_{ij}(x)|\leqslant d\,\Big(\max_{m,n} \big|\partial_k \mathbf{D}_{mn}(x)\big|\Big),
\end{equation}
where \( 1 \leq m, n \leq d \). Using the first-order regularity hypothesis on $\mathbf{D}$, together with equations~\eqref{Bound-1}, \eqref{Bound-2}, and~\eqref{Bound-3}, we obtain the final estimate
\begin{equation}
    |\partial_k \sigma_{ij}(x)|\leqslant \frac{d^2}{2\sqrt{\alpha}}\,\Big(\max_{m,n} \big|\partial_k \mathbf{D}_{mn}(x)\big|\Big) \leqslant  \frac{d^2}{2\sqrt{\alpha}} M_k.
\end{equation}
Since it holds for all $x\in\mathbb{R}^d$ and all \( i,j,k \in \{1, \ldots, d\} \), this completes the proof.
\end{proof}

\begin{remark}\label{regularity2}
Clearly, hypothesis \((\mathrm{H1})\) implies the first-order regularity assumption in the statement of Lemma~\ref{regularity}.
\end{remark}

\begin{lemma}[Semigroup generation]\label{semigroup}
If the domain of $\mathcal{L}_0$ is defined as $D\left(\mathcal{L}_0\right):=C_{c}^{\infty}(\mathbb{R}^d)$ and assumptions (H1) and (H2) are satisfied, then the closure of the convection-diffusion operator $\mathcal{L}_0$ generates a strongly continuous contraction semigroup $\mathcal{T}(t)$ on $L^1(\mathbb{R}^d)$.
\end{lemma}

\begin{proof}
The operator $\mathcal{L}_0$ has the standard Fokker-Planck form
\begin{equation}\label{FPO-2}
\mathcal{L}_0 \, f = -\sum_{i=1}^d \partial_i (\tilde{b}_i f) + \frac{1}{2} \sum_{i,j=1}^d \partial_i \partial_j (\mathbf{D}_{ij} f), \qquad \textrm{for all}\qquad f\in C_c^{\infty}(\mathbb{R}^d).
\end{equation}
From assumptions~(H1) and~(H2) along with Lemma~\ref{regularity} and Remark~\ref{regularity2}, and for some constant $K^{\prime} > 0$, we have that
    $$
    \big|\partial_k \tilde{b}_i(x)\big| + \big|\partial_k \sigma_{ij}(x)\big| \leq K^{\prime},
    $$
for all $x \in \mathbb{R}^d$ and all $i, j, k = 1, \ldots, d$. 
Therefore, applying Theorem~3.1 and Lemma~3.2 of~\cite{chen2018numerical} in~\eqref{CDE0}, we can conclude that $\left(\mathcal{L}_0,D\left(\mathcal{L}_0\right)
\right)$ generates a strongly continuous contraction semigroup on $(L^1(\mathbb{R}^d),||\cdot||_1)$, which is denoted by $\mathcal{T}(t)$.
\end{proof}

\begin{definition}
 Let us denote by $C\big([0,T];L^1(\mathbb{R}^d)\big)$ the space of all functions $u(x,t)$ such that $\sup_{t \in [0,T]} \|u(\cdot, t)\|_1 < \infty$ and, for all  $t \in [0,T], \|u(\cdot, s) - u(\cdot, t)\|_1 \longrightarrow 0$ as  $s \longrightarrow t$.
\end{definition}

\begin{definition}
A function $u\in C\big([0,T];L^1(\mathbb{R}^d)\big)$ is called a mild solution of equation~\eqref{CDE0} if
\begin{equation*}
     u(\cdot, t)= \int\limits_0^t \overline{\mathcal{L}_0}\,
     u(\cdot, s)\,\mathrm{d}s+u_0
     \qquad \text{in $L^1(\mathbb{R}^d)$ for all $0 \le t \le T$},
\end{equation*}
where $\overline{\mathcal{L}_0}$ is the closure of $\mathcal{L}_0$ in $L^1(\mathbb{R}^d)$ and the integral is a Bochner integral in $L^1(\mathbb{R}^d)$, possibly improper at the origin.
\end{definition}

\begin{theorem}[Relation between PDE and SDE]
\label{theorem-main}
Under assumptions \((\mathrm{H1})-(\mathrm{H3})\), the SDE~\eqref{SDE-ITO} admits a unique strong solution \(X_t\) on $[0,T]$ for any $T>0$. Furthermore, the convection-diffusion equation~\eqref{opeEquation} has a unique solution \(u \in C^1((0, T]; C(\mathbb{R}^d)) \cap C((0, T]; C^2(\mathbb{R}^d)) \cap C([0, T]; L^1(\mathbb{R}^d))\). In addition, this solution \(u(\cdot, t)\) of equation~\eqref{opeEquation} corresponds to the PDF of \(X_t\).
\end{theorem}

\begin{proof}
Assumptions \((\mathrm{H1})-(\mathrm{H3})\), by Remark~\ref{globalcont} in the present work, include the assumptions for the existence and uniqueness of a strong solution $X_t$ of Itô SDE~\eqref{SDE-ITO} as can be found in~\cite{protter2005stochastic} (see Theorem~7 in Chapter~V). On the other hand, from our Lemma~\ref{semigroup}, we know that the closure of $\mathcal{L}_0$ is the generator of a strongly continuous semigroup $\mathcal{T}(t)$ on $(L^1(\mathbb{R}^d),||\cdot||_1)$. Therefore, we define the mapping
\begin{equation*}
    u:t\longmapsto u(\cdot,t):=\mathcal{T}(t)u_0,
\end{equation*}
from where it follows that $u\in C\big([0,T];L^1(\mathbb{R}^d)\big)$ due to the strong continuity of $\mathcal{T}(t)$. Then, from Proposition II.6.4 in~\cite{engel2000one}, it follows that $u$ is the unique mild solution of~\eqref{CDE0}. Moreover, by Remark~2.5 in~\cite{chen2018numerical}, this solution is classical, thus \(u \in C^1((0, T]; C(\mathbb{R}^d)) \cap C((0, T]; C^2(\mathbb{R}^d)) \cap C([0, T]; L^1(\mathbb{R}^d))\). Besides, by our Lemma~\ref{lem:sfpf}, it is the unique solution to initial value problem~\eqref{opeEquation}. As a final step to conclude the proof, we use Theorem~2.4 in~\cite{chen2018numerical} to find that $u(\cdot,t)$ is the PDF of $X_t$.
\end{proof}

\section{Multidimensional Definition of the H\"anggi-Klimontovich Integral}
\label{sec:multidimensional}

The H\"anggi-Klimontovich integral is the name employed in part of the physical literature to denote a variant of the Stratonovich integral. The difference from the latter lies in the evaluation point of the integrand in each subinterval of the partition in the Riemann sum approximation: instead of using the midpoint, the right endpoint is selected. Although the mathematical structure of this integral was already examined in~\cite{escudero2023ito}, this section is devoted to analyze those of its properties that will be needed in the remainder of this work.

\begin{definition} An $\mathbb{R}^d\,$-valued Markov process $\{Y_t,\, 0\leq t \leq T\}$ is called a \emph{diffusion process} if it satisfies the following three conditions for every $t\in[0,T]$ and $x \in \mathbb{R}^d$, and any $\delta>0$:

\begin{enumerate}
	\item Continuity of sample paths:
	$$ \lim_{h\to0^+}\frac{1}{h}\, \mathbb{P} \,\, \Big({\big|Y_{t+h}-Y_t\big|>\delta}\,\big| Y_t=x\Big)=0.$$
	\item Existence of drift vector $b:\mathbb{R}^d\times[0,T]\longrightarrow \mathbb{R}^d$ such that
	$$\lim_{h\to0^+}\frac{1}{h}\, \mathbb{E} \,\,\Big(Y_{t+h}-Y_t\,\big| Y_t=x\Big)=b(x, t).$$
	\item Existence of diffusion matrix $\mathbf{D}:\mathbb{R}^d\times[0,T]\longrightarrow \mathcal{M}_{d\times d}(\mathbb{R})$ such that 
	$$\lim_{h\to0^+}\frac{1}{h}\, \mathbb{E} \,\,\Big((Y_{t+h}-Y_t\big)(Y_{t+h}-Y_t\big)^\top| Y_t=x\Big)=\mathbf{D}(x,t).$$
\end{enumerate}
\end{definition}

From now on, we will take $\mathbf{D}(x,t)$ to be uniformly elliptic (see the previous section) and $\sigma(x,t):\mathbb{R}^d\times[0,T] \longrightarrow \mathcal{M}_{d\times d}(\mathbb{R})$ to be the principal square root of $\mathbf{D}(x,t)$, and we will refer to this function as the \emph{amplitude of noise}.

\begin{definition} 	Let $b:\mathbb{R}^d\times[0,T] \longrightarrow \mathbb{R}^d$ and $\sigma(x,t):\mathbb{R}^d\times[0,T] \longrightarrow \mathcal{M}_{d\times d}(\mathbb{R})$ be two measurable functions. 

	\begin{itemize}
		\item (Lipschitz condition) It is said that $b$ and $\sigma$ satisfy the Lipschitz condition, in the first argument, if there exists a constant $C>0$ such that for all $x, y \in \mathbb{R}^d$ and $0 \leqslant t \leqslant T,$
		$$|b(x,t)-b(y,t)|\leqslant C|x-y|,\qquad ||\sigma(x,t)-\sigma(y,t)||_{F} \leqslant C|x-y|.$$
		\item (Linear growth condition). It is said that $b$ and $\sigma$ satisfy the linear growth condition, in the first argument, if there exists a constant $C'>0$ such that for all $x \in \mathbb{R}^d$ and $0 \leqslant t \leqslant T,$
		$$|b(x,t)| \leqslant C'(1+|x|),\qquad ||\sigma(x,t)||_{F} \leqslant C'(1+|x|).$$
	\end{itemize}
\end{definition}

\begin{definition}\label{def_HKI_multi}
Assume that $b(x,t)$ and $\sigma(x,t)=\big(\sigma_{ij}(x,t)\big)$ are continuous on $\mathbb{R}^d\times[0,T]$ and both satisfy the Lipschitz and linear growth conditions in $x$. Furthermore, consider a $d$-dimensional diffusion process $({Y}_t)_{t\geq 0}=(Y^{1}_t,\ldots,Y^{d}_t )_{t \geq 0}$, with drift vector $b(x,t)$ and noise amplitude $\sigma(x,t)$. For any $\Psi:\mathbb{R}^d\times[0,T]\longrightarrow\mathcal{M}_{d\times d}(\mathbb{R})$, for which all its partial derivatives $\partial_{i}\Psi(x,t)$ are continuous, the H\"anggi-Klimontovich integral of $\big(\Psi({Y}_t,t)\big)_{t \geq 0}=\big[\Psi_{ij}({Y}_t,t)\big]_{t \geq 0}$ in the interval $[0,T]$ with respect to the diffusion process $({Y}_t)_{\geq 0}$ is denoted by
\begin{equation}\label{HKI_multi}
    \int\limits_0^T \Psi(Y_t,t)\bullet \mathrm{d}Y_t
\end{equation}
and defined to be the $\mathbb{R}^d\,$-valued random variable whose $i$-th component is given by
\begin{equation}\label{def_HKI_i}
		\int\limits_0^T \Psi_{i*}(Y_t,t)\bullet Y_t :=\lim_{||\Delta_n||\to 0}\sum_{j=1}^{n}\Psi_{i*}(Y_{t_j},t_j)(Y_{t_j}-Y_{t_{j-1}}) \quad \text{in probability},
\end{equation}
where $\Delta_n=\{t_0, t_1, \ldots, t_{n-1}, t_n\}$ is a partition of $[0, T]$, and $||\Delta_n||=\max_{1\leq k \leq n}(t_k-t_{k-1})$.
\end{definition}
This integral is well-defined and its relation with the Itô integral is expressed in the following theorem.

\begin{theorem}\label{Theorem_RCM}
	Under the assumptions established in Definition~\ref{def_HKI_multi}, the limit in equation~\eqref{HKI_multi} exists. Furthermore, for $\mathbf{D}=\sigma \sigma^{\top}$, it is connected to the Itô integral through the following identity
	\begin{equation}\label{RCM}
		\int\limits_0^T \Psi(Y_t,t)\bullet \mathrm{d}Y_t=\int\limits_0^T \Psi(Y_t,t)\,\mathrm{d}Y_t+\sum_{l,k=1}^{d}\int\limits_0^T\big(\partial_{k}\Psi(Y_t,t)\big)_{* l} \mathbf{D}_{lk}(Y_t,t)\,\mathrm{d}t,
	\end{equation}
which holds almost surely. 
\end{theorem}

The proof of Theorem~\ref{Theorem_RCM} can be found in~\cite{escudero2023ito}.

\begin{proposition}\label{sigma}
Let $X_t$ be a $d$-dimensional diffusion process with drift $b(x,t)$, noise amplitude $\sigma(x,t)$, and such that $X_0 =x_0 \in \mathbb{R}^d$; let also $W_t$ be a $d$-dimensional Brownian motion. Assume that $b(x,t)$ and $\sigma(x,t)$ obey the same assumptions as in Definition~\ref{def_HKI_multi} and, moreover, that all partial derivatives $\partial_{i}\sigma(x,t)$ are continuous. Then the identity 
\begin{equation}\label{equi_main}
		\int\limits_0^T \sigma(X_t,t)\bullet \mathrm{d}W_t=\int\limits_0^T \sigma(X_t,t)\,\mathrm{d}W_t+\sum_{k, l=1}^{d}\int\limits_0^T\big(\partial_{k}\sigma(X_t,t)\big)_{* l} \sigma_{kl}(X_t,t)\,\mathrm{d}t
	\end{equation}
holds almost surely.
\end{proposition}

\begin{proof}
Choose $Y_t=(X_t,W_t)$ and $\Psi(x,t)=(\mathbf{0},\sigma(x,t))$, where $\mathbf{0}$ is a null matrix with dimension $d \times d$, such that
\begin{equation}\label{equi1}
	\int\limits_{0}^{T} \Psi(Y_t,t)\bullet \mathrm{d}Y_t=\int\limits_{0}^{T}\big(\mathbf{0}, \sigma(X_t,t)\big)\bullet \mathrm{d} \begin{pmatrix} X_t  \\ W_t \end{pmatrix}=\int\limits_{0}^{T} \sigma(X_t,t)\bullet \mathrm{d}W_t.
\end{equation}
Under the stated assumptions, the process $X_t$ can be written in the form~\cite{stroock1979}:

\begin{equation}    X_t=x_0+\int\limits_{0}^{T}u(X_t,t)\,\mathrm{d}t + \int\limits_{0}^{T}\sigma(X_t,t)\,\mathrm{d}W_t.
\end{equation}
In this case, by equations~(10) and~(11) in~\cite{stratonovich}, the diffusion coefficient of the process $Y_t$ is 
\begin{equation}
    A(x,t)=\big[a_{ij}(x,t)\big]= \begin{bmatrix}
\mathbf{D} & \sigma \\
\sigma^\top & \mathrm{I}
\end{bmatrix},
\end{equation}
where $\mathrm{I}$ is the identity matrix with dimension $d \times d$. Therefore, the transformation formula~\eqref{RCM} of Theorem~\ref{Theorem_RCM} applied to equation~\eqref{equi1} yields
\begin{equation}\label{equi2}
\int\limits_{0}^T \sigma(X_t, t)\bullet \mathrm{d}W_t = \int\limits_{0}^T \sigma(X_t,t)\, \mathrm{d}W_t +\sum_{l=1}^{2d}\sum_{k=1}^{2d} \int\limits_{0}^T \big(\mathbf{0},\partial_{k}\sigma(X_t,t)\big)_{*l}\,a_{lk}\,\mathrm{d}t.
\end{equation}
Note that $\partial_{k}\sigma(x,t)=\big[\partial_{k}\sigma_{ij}\big]$ is a null matrix for $ d+1\leq k \leq 2d$. Furthermore, the vector $\big(\partial_{k}\Psi\big)_{*l}=\big(\mathbf{0},\partial_{k}\sigma(X_t,t)\big)_{*l}$ is a null vector for $ 1\leq l \leq d$. Finally, replace these two observations in equation~\eqref{equi2} to find that the identity~\eqref{equi_main} is verified.
\end{proof}

\section{H\"anggi-Klimontovich Calculus}
\label{hkcalculus}

As one of our present goals is to connect the integral introduced in the previous section with PDEs, herein we perform an intermediate step: we connect it with the Itô integral in the context of stochastic differential equations. We start introducing the concept of solution.

\begin{definition}
We say that a $d$-dimensional diffusion process $X_t$ that satisfies a stochastic integral equation of the form
\begin{equation}
X_t= Z +\int\limits_{0}^t b(X_s, s)\,\mathrm{d}s+\int\limits_{0}^t \sigma(X_s, s)\bullet\mathrm{d}W_s,
\end{equation}
for $0 \le t \le T$, some positive real number $T$, some $\mathcal{F}_0$-measurable random variable $Z$ whose distribution is absolutely continuous with respect to Lebesgue measure, and some functions $b$ and $\sigma$ such that all the expressions are well defined,
is a solution to the H\"anggi-Klimontovich stochastic differential equation (HK SDE)
\begin{equation}\label{HK-SDE}
\mathrm{d}X_t= b(X_t, t)\,\mathrm{d}t+\sigma(X_t, t)\bullet\mathrm{d}W_t, \qquad X_0= Z.
\end{equation}
\end{definition}

\begin{remark}
As in Section~\ref{sec:cdequations}, we take $\{\mathcal{F}_t\}_{t \ge 0}$ to be the initially enlarged filtration $\mathcal{F}_t=\mathcal{F}_t^W \vee \sigma(Z)$ and $Z$ to be independent of $\{W_t\}_{t \ge 0}$.
\end{remark}

\begin{theorem}[Conversion rule for SDEs]\label{proptrans}
Let $b
:\mathbb{R}^d\times[0, T] \longrightarrow \mathbb{R}^d$ and  $\sigma:\mathbb{R}^d\times[0, T] \longrightarrow \mathcal{M}_{d\times d}(\mathbb{R})$ be functions as those specified in Proposition~\ref{sigma}. Furthermore, assume that the function $h:=\nabla\cdot\mathbf{D}-\sigma\,\nabla\cdot\sigma^\top$ satisfies the Lipschitz and linear growth conditions in $x$. Then, the unique solution to the It\^o SDE
\begin{equation}\label{transformation_SDE}
\mathrm{d}X_t= \Big[b(X_t, t)+h(X_t,t)\Big]\,\mathrm{d}t+ \sigma(X_t, t)\,\mathrm{d}W_t, \qquad X_0=Z,
\end{equation}
solves equation~\eqref{HK-SDE} almost surely. Correspondingly, there almost surely exists a unique solution to~\eqref{HK-SDE} that is given
by the solution to~\eqref{transformation_SDE}. 
\end{theorem}

\begin{proof}
Under the stated assumptions, the It\^o SDE~\eqref{transformation_SDE} has a unique solution~\cite{protter2005stochastic}. Now, to prove the equivalence between SDEs~\eqref{HK-SDE} and~\eqref{transformation_SDE}, we write the identity~\eqref{equi_main} using the stochastic differential notation
\begin{equation}\label{convert_1}
		 \sigma(X_t,t)\bullet \mathrm{d}W_t= \sigma(X_t,t)\,\mathrm{d}W_t+\Bigg[\sum_{k, l=1}^{d}\big(\partial_{k}\sigma(X_t,t)\big)_{* l} \sigma_{kl}(X_t,t)\Bigg]\,\mathrm{d}t.
	\end{equation}
If we denote by $h$ the factor inside the square brackets on the right hand side of~\eqref{convert_1}, then its $i$-th component is given by
\begin{eqnarray}\nonumber
[h]_i&=&\sum_{l=1}^{d}\sum_{k=1}^{d}\big(\partial_k\sigma_{il}\big)\sigma_{kl}\\ \nonumber
    &=&{\big[\nabla\,\sigma:\sigma^\top\big]}_i,
\end{eqnarray}
from where we obtain the following equality
\begin{equation}\label{convert_2}
    h(x,t)=\nabla\sigma:\sigma^\top.
\end{equation}
On the other hand, we have 
\begin{eqnarray}\nonumber
[\nabla\cdot\mathbf{D}]_i&=&\big[\nabla\cdot\big(\sigma \sigma^\top\big)\big]_i\\ \nonumber &=&\sum_{k=1}^{d}\Bigg( \partial_k\bigg[\sum_{l=1}^{d}\sigma_{il}\sigma_{kl}\bigg]_{ik}\Bigg)\\ \nonumber
&=&\sum_{l=1}^{d}\sum_{k=1}^{d}\big(\partial_k\sigma_{il}\big)\sigma_{kl}+\sum_{l=1}^{d}\sigma_{il}\Bigg(\sum_{k=1}^{d}\partial_k\sigma_{kl}\Bigg)\\ \nonumber
  &=&{\big[\nabla\sigma:\sigma^\top\big]}_i+{\big[\sigma\,\nabla\cdot\sigma^\top\big]}_i,
\end{eqnarray}
so we have another equality
\begin{equation}\label{convert_3}
\nabla\cdot\mathbf{D}= \nabla\sigma:\sigma^\top + \sigma\,\nabla\cdot\sigma^\top.
\end{equation}
Equation~\eqref{transformation_SDE} is easily deduced from identities~\eqref{convert_1}, \eqref{convert_2}, and~\eqref{convert_3}. Therefore, we conclude that the unique solution of~\eqref{transformation_SDE} solves~\eqref{HK-SDE} almost surely. To check uniqueness assume there exist two solutions; by the same formula this enters in contradiction with the uniqueness of solution to equation~\eqref{transformation_SDE}.
\end{proof}

\begin{corollary}\label{Corollary-main} 
Under the hypotheses established in Theorem~\ref{proptrans} and the additional assumption $\,\sigma\,{\nabla\cdot\sigma}^\top=\frac{1}{2}\nabla\cdot\mathbf{D}$, there exists a unique solution to the It\^o SDE
\begin{equation}\label{transformation_SDE_simple}
\mathrm{d}X_t= \bigg[b(X_t, t)+\frac{1}{2}\nabla\cdot\mathbf{D}(X_t,t)\bigg]\,\mathrm{d}t+ \sigma(X_t, t)\,\mathrm{d}W_t, \qquad X_0=Z
\end{equation}
that moreover solves equation~\eqref{HK-SDE} almost surely. Correspondingly, there almost surely exists a unique solution to~\eqref{HK-SDE} that is given by the solution to~\eqref{transformation_SDE_simple}.
\end{corollary}

\section{Convection-diffusion equations and HK calculus}\label{sec:cdehkc}

Now we state one of our main results, which connects HK SDEs with convection-diffusion PDEs. However, that connection is not general and relies on a very particular structure of the diffusion coefficient, as the following statement shows.

\begin{theorem}\label{teo-HK-FPE}
    Let $X_t$ be a $d$-dimensional diffusion process described by the autonomous HK SDE
\begin{equation}
   \mathrm{d}X_t= b(X_t)\,\mathrm{d}t+\sigma(X_t)\bullet\mathrm{d}W_t, \qquad X_0=Z,
\end{equation}
where the drift $b$, the diffusion coefficient $\mathbf{D}=\sigma{\sigma}^{\top}$, and the initial condition $Z$ obey the hypotheses \((\mathrm{H1})-(\mathrm{H3})\). Furthermore, assume that $\sigma$ is the principal square root of $\mathbf{D}$ and the identity $\nabla\cdot\mathbf{D}=2\,\sigma\,\nabla\cdot\sigma^\top$ holds. Then, the equation 
   \begin{equation}\label{finaly-equation}
   \partial_t u= - \nabla \cdot\big(b({x})\,u\big)+\frac{1}{2}\nabla \cdot\big(\mathbf{D}({x})\nabla u\big),   
   \end{equation}
with initial condition $u(x, 0)=u_0(x) \in L^1(\mathbb{R}^d)$ such that $Z \sim u_0(\cdot)$, possesses a unique mild solution \(u(x, t)\), which is in turn classical and corresponds to the probability density function of \(X_t\).
\end{theorem}

\begin{proof}
Hypotheses \((\mathrm{H1})-(\mathrm{H3})\) imply, by Lemma~\ref{regularity}, the regularity assumptions in Theorem~\ref{proptrans}. This, together with the structural condition $\nabla\cdot\mathbf{D}=2\,\sigma\,\nabla\cdot\sigma^\top$, implies the result of Corollary~\ref{Corollary-main}.
The statement then follows as an immediate consequence of this Corollary along with Theorem~\ref{theorem-main}.
\end{proof}

Of course, the equality $\nabla\cdot\mathbf{D}=2\,\sigma\,\nabla\cdot\sigma^\top$ holds trivially in one dimension but does not hold generically in any higher dimension due to the non-commutative character of the matrix algebra. Therefore, in the following, we illustrate this structural condition with a series of examples in which it either does or does not hold. We anticipate that such a condition is rather restrictive and then fails in general. We can only expect it to hold for highly structured matrices.

\subsection{Positive examples}\label{positive-cases}
In this subsection we list a number of cases in which the matrix $\sigma$ does satisfy the identity 
\begin{equation}\label{identity-key}  \nabla\cdot\mathbf{D}=2\,\sigma\,\nabla\cdot\sigma^\top,  
\end{equation}
which is the structural condition assumed in Theorem~\ref{teo-HK-FPE}.

\subsubsection*{\textbf{Case~1: Constant diffusion tensor}}
Assume \(\mathbf{D}:\mathbb{R}^d\to\mathcal{M}_{d\times d}(\mathbb{R})\)  is constant and positive definite, \(\mathbf{D}(x)=D_0\) for all \(x\in\mathbb{R}^d\). Let \(\sigma(x)=D_0^{1/2}\) denote its principal square root, so that \(\mathbf{D}=\sigma\sigma^{\top}\). Then, for all indices \(i,j,k\) and all \(x\in\mathbb{R}^d\), we obtain 
\[
\partial_k \mathbf{D}_{ij}(x)=0,\qquad \partial_k \sigma_{ij}(x)=0,
\]
which meets the boundedness requirements in~(H1). Furthermore, hypothesis~(H2) is satisfied since $\mathbf{D}$ is positive definite. Hence, 
\[
(\nabla\!\cdot \mathbf{D})_i(x)=\sum_{j=1}^d \partial_j \mathbf{D}_{ij}(x)=0,\qquad
(\nabla\!\cdot \sigma^{\top})_i(x)=\sum_{j=1}^d \partial_j \sigma_{ji}(x)=0,
\]
and therefore, for all \(x\in\mathbb{R}^d\), the identity~\eqref{identity-key} is verified.

\subsubsection*{\textbf{Case~2: Isotropic diffusion tensor}}
Let \(\mathbf{D}:\mathbb{R}^d\to\mathcal{M}_{d\times d}(\mathbb{R})\) be an isotropic tensor such that \(\mathbf{D}(x)=g(x)\,\mathbb{I}_d\), where \(g\in C_b^{5}(\mathbb{R}^d)\) and \(\inf_{x\in\mathbb{R}^d} g(x)\ge \alpha>0\).
Furthermore, let \(\sigma(x)=\sqrt{g(x)}\,\mathbb{I}_d\) be its principal square root. Since \(g\) belongs to \(C_b^{5}\), the smoothness and derivative bounds required in~(H1) hold. Since \(g\) is bounded away from zero, the uniform ellipticity~(H2) follows from \(\xi^{\top}\mathbf{D}(x)\,\xi=g(x)|\xi|^2\ge \alpha|\xi|^2\) for every \(\xi\in\mathbb{R}^d\). Given the diagonal structure of $\mathbf{D}(x)$, we obtain
\[
(\nabla\!\cdot \mathbf{D})_i(x)=\sum_{j=1}^d \partial_j \mathbf{D}_{ij}(x)=\partial_i g(x),
\qquad
(\nabla\!\cdot \sigma^{\top})_i(x)=\sum_{j=1}^d \partial_j \sigma_{ji}(x)=\partial_i \sqrt{g(x)}=\frac{\partial_i g(x)}{2\sqrt{g(x)}},
\]
which ensures that the structural condition~\eqref{identity-key} is fulfilled uniformly.

\subsubsection*{\textbf{Case~3: Diagonal anisotropic diffusion tensor}}
Let \(g_k\in C_b^{5}(\mathbb{R}^d)\) be such that \(\inf_{x\in\mathbb{R}^d} g_k(x)\ge \alpha_k>0\) for each \(k=1,\dots,d\). Define
\[
\mathbf{D}(x)=\mathrm{diag}\big(g_1(x),\dots,g_d(x)\big).
\]
Let \(\sigma\) denote its principal square root; then,
\[
\sigma(x)=\mathrm{diag}\big(\sqrt{g_1(x)},\dots,\sqrt{g_d(x)}\big).
\]
Analogously to Case~2, the assumptions on \(g_k\) imply that all required derivatives of \(\mathbf{D}\) up to order \(5\) are bounded, and uniform ellipticity holds with
\(\alpha:=\min_{1\le k\le d}\alpha_k>0\). Similarly to the previous case, we obtain
\[
(\nabla\!\cdot \mathbf{D})_i(x)=\sum_{j=1}^d \partial_j \mathbf{D}_{ij}(x)=\partial_i g_i(x),
\qquad
(\nabla\!\cdot \sigma^{\top})_i(x)=\sum_{j=1}^d \partial_j \sigma_{ji}(x)
=\partial_i \sqrt{g_i(x)}=\frac{\partial_i g_i(x)}{2\sqrt{g_i(x)}}.
\]
Consequently,
\[
\big(2\,\sigma\,(\nabla\!\cdot \sigma^{\top})\big)_i(x)
=2\,\sqrt{g_i(x)}\,\frac{\partial_i g_i(x)}{2\sqrt{g_i(x)}}
=\partial_i g_i(x)
=(\nabla\!\cdot D)_i(x),
\]
which ensures that the structural condition~\eqref{identity-key} is met.

\subsubsection*{\textbf{Case~4: Orthogonally diagonalizable diffusion tensor}}
Let $\mathbf{D}:\mathbb{R}^d\to \mathcal{M}_{d\times d}(\mathbb{R})$ be positive definite such that $\mathbf{D}\in C_b^{5}(\mathbb{R}^d)$. Suppose that there exist $R\in\mathbb{O}(d)$ constant and functions $g_k\in C_b^{5}(\mathbb{R}^d)$ where $\inf_{x\in\mathbb{R}^d} g_k(x)\ge \alpha_k>0$ such that
\[
\mathbf{D}(x)=R\,\mathrm{diag}\big(g_1(x),\dots,g_d(x)\big)\,R^{\top}.
\]
Let
\[
\sigma(x)=R\,\mathrm{diag}\big(\sqrt{g_1(x)},\dots,\sqrt{g_d(x)}\big)\,R^{\top}
\]
be its principal square root, so that $\mathbf{D}=\sigma\sigma^{\top}$.
Since $R$ is constant and $g_k\in C_b^{5}(\mathbb{R}^d)$ with $g_k\ge \alpha_k>0$, the smoothness requirements in~(H1) hold and the uniform ellipticity in~(H2) follows with $\alpha:=\min_{1\le k\le d}\alpha_k$. Let $r_k$ denote the $k$-th column of $R$; then we obtain
\[
\nabla\!\cdot \mathbf{D}=\sum_{k=1}^d r_k\,\big((\nabla g_k)^{\top}\, r_k\big),\qquad
\nabla\!\cdot \sigma^{\top}=\sum_{k=1}^d r_k\,\big((\nabla \sqrt{g_k})^{\top}\, r_k\big),
\]
and
\[
2\,\sigma\,\nabla\!\cdot \sigma^{\top}
=\sum_{k=1}^d 2\sqrt{g_k}\,r_k\,\big((\nabla \sqrt{g_k})^{\top}\, r_k\big)
=\sum_{k=1}^d r_k\,\big((\nabla g_k)^{\top}\, r_k\big)
=\nabla\!\cdot \mathbf{D},
\]
which verifies identity~\eqref{identity-key}. Note that this case reduces to Case~3 when $R=\mathbb{I}_d$ and that any matrix is orthogonally diagonalizable if and only if it is symmetric (by the spectral theorem).

\begin{remark}\label{constant-rotations} The structural condition might hold even if we considered a non-symmetric noise amplitude. This can be shown by applying a constant orthogonal rotation in Case~4: consider a constant matrix $Q\in\mathbb{O}(d)$ and define
\[
\sigma(x)=\mathbf{D}^{1/2}(x)\,Q,
\]
which is generally non-symmetric and satisfies $\mathbf{D}=\sigma\sigma^{\top}$. In this case, using the symmetry of $\mathbf{D}^{1/2}$ and the constancy of $Q$, we obtain 
\[
\,\sigma\,(\nabla\!\cdot \sigma^{\top})=\mathbf{D}^{1/2}\,Q\,\nabla\!\cdot\!\big(Q^{\top}\mathbf{D}^{1/2}\big)=\mathbf{D}^{1/2}\,\nabla\!\cdot\!\mathbf{D}^{1/2},
\]
which, together with the result obtained in Case~4, verifies identity~\eqref{identity-key}. 
\end{remark}

\begin{remark} It is even possible to go beyond constant rotations upon the inclusion of suitable assumptions on the rotation matrix. So, in contrast to what was presented in Remark~\ref{constant-rotations}, if we set
\[
\sigma(x)=\mathbf{D}^{1/2}(x)\,Q(x), \qquad Q:\mathbb{R}^d\to\mathbb{O}(d),\ \ Q\in C_b^{5}(\mathbb{R}^d);
\]
then, the identity~\eqref{identity-key} holds if and only if
\[
\mathbf{D}^{1/2}\,Q\,(\nabla\!\cdot Q^{\top})\,\mathbf{D}^{1/2}=0
\quad\text{for all } x\in\mathbb{R}^d.
\]
Two immediate sufficient conditions are: $Q$ constant, which reproduces the case presented in Remark~\ref{constant-rotations}, or $\nabla\!\cdot Q^{\top}=0$ on $\mathbb{R}^d$, which means that the rows of $Q$ are divergence-free, that is, the rows of $Q$ are solenoidal vector fields. Outside these classes, the equality above typically fails, and identity~\eqref{identity-key} does not hold.
\end{remark}

\subsubsection*{\textbf{Case~5: Diffusion in the presence of a privileged orientation}}
Consider now the case in which diffusion takes place in an anisotropic medium in which a single privileged direction exists, such as a fiber-aligned diffusion.
Let $v\in\mathbb{R}^d$ with $\|v\|=1$ be the privileged orientation and $f,g\in C_b^5(\mathbb{R}^d)$ with
$\inf_x f(x)\ge \alpha_1>0$ and $\inf_x g(x)\ge \alpha_2>0$. Let us consider
\[
\mathbf D(x)=f(x)\,vv^\top+g(x)\,(\mathbb{I}_d-vv^\top)
= g(x)\,\mathbb{I}_d+\big(f(x)-g(x)\big)vv^\top .
\]
Since $v$ is constant and $f,g\in C_b^5$, all spatial derivatives of $\mathbf D$ up to order five are continuous and bounded, hence~(H1) holds. For uniform ellipticity, we use the existence and uniqueness of the orthogonal decomposition relative to $\mathrm{span}\{v\}$: for every $\xi\in\mathbb{R}^d$ there exist unique vectors $\xi_1,\xi_2$ with
\[
\xi=\xi_1+\xi_2,\qquad \xi_1=(v^\top\xi)\,v,\qquad \xi_2\perp v,\qquad |\xi|^2=|\xi_1|^2+|\xi_2|^2.
\]
Since $vv^\top$ and $I_d-vv^\top$ are complementary orthogonal projectors (as $\|v\|=1$), we obtain 
\[
\xi^\top \mathbf D(x)\,\xi
= f(x)\,|\xi_1|^2 + g(x)\,|\xi_2|^2
\;\ge\; \alpha_1\,|\xi_1|^2 + \alpha_2\,|\xi_2|^2
\;\ge\; \alpha\,|\xi|^2,
\qquad \alpha:=\min\{\alpha_1,\alpha_2\}>0,
\]
so $\mathbf D(x)$ is positive definite for every $x$ and~(H2) holds with ellipticity constant $\alpha$. Subsequently, we diagonalize $\mathbf D$. Since
\[
\mathbf D(x)=g(x)\,\mathbb{I}_d+\big(f(x)-g(x)\big)vv^\top,
\]
we have $\mathbf D(x)\,v=f(x)\,v$, and for any $w\perp v$, $\mathbf D(x)\,w=g(x)\,w$. Thus, the spectrum is
\[
\{\,f(x)\ \text{(with multiplicity\ 1)},\ g(x)\ \text{(with multiplicity\ }d-1)\,\}.
\]
Choose an orthonormal basis $\{v,e_2,\ldots,e_d\}$ with $e_j\perp v$, $j=2,\cdots,d$, and let $R\in \mathbb{O}(d)$ be the constant orthogonal matrix whose columns are $v,e_2,\ldots,e_d$. Then
\[
\mathbf D(x)=R\,\mathrm{diag}\big(f(x),\,g(x),\ldots,g(x)\big)\,R^\top .
\]
The principal square root is
\[
\sigma(x)=R\,\mathrm{diag}\big(\sqrt{f(x)},\,\sqrt{g(x)},\ldots,\sqrt{g(x)}\big)\,R^\top,
\qquad \mathbf{D}=\sigma\sigma^{\top}.
\]
Proceeding analogously to Case~4, since $R$ is constant and $\Lambda(x):=\mathrm{diag}(\sqrt{f(x)},\sqrt{g(x)},\ldots,\sqrt{g(x)})$ is diagonal, we have componentwise $\nabla\!\cdot \Lambda^2=2\,\Lambda\,(\nabla\!\cdot \Lambda^\top)$. Therefore,
\[
\nabla\!\cdot \mathbf D=2\,\sigma\,\nabla\!\cdot \sigma^\top,
\]
and the structural identity~\eqref{identity-key} holds pointwise.

\subsubsection*{\textbf{Case~6: Scalar-modulated constant anisotropic diffusion}}
Let $B\in\mathcal{M}_{d\times d}(\mathbb{R})$ be constant, symmetric, and positive definite.
Let $g\in C_b^{5}(\mathbb{R}^d)$ with $g(x)\ge \alpha>0$ for every $x\in\mathbb{R}^d$. Consider
\[
\mathbf{D}(x)=g(x)\,B.
\]
Since $B$ is constant and $g$ belongs to $C_b^{5}$, all spatial derivatives of $\mathbf{D}$ up to order five are continuous and bounded, which meets the requirement in~(H1).
Since $B$ is positive definite, there exists a constant $c_B>0$  such that
\[
\xi^{\top}B\,\xi \ge c_B \,|\xi|^2 \qquad \text{for all}\quad \xi\in\mathbb{R}^d.
\]
Additionally, since $g(x)\ge \alpha$, we obtain
\[
\xi^{\top}\mathbf{D}(x)\,\xi \ge \alpha\,c_B\,|\xi|^2 \qquad \text{for every } x \in\mathbb{R}^d.
\]
Then $\mathbf{D}(x)$ is positive definite for all $x$ and the uniform ellipticity assumption~(H2) holds with constant $\alpha\,c_B$. Take the principal square root $\sigma(x)=\sqrt{g(x)}\,B^{1/2}$, so that $\mathbf{D}=\sigma\sigma^{\top}$.
Since $B$ is constant,
\[
\nabla\!\cdot \mathbf{D}(x)=B\,\nabla g(x),
\qquad
(\nabla\!\cdot \sigma^{\top})(x)=B^{1/2}\,\nabla\sqrt{g(x)}.
\]
Hence,
\[
2\,\sigma(x)\,(\nabla\!\cdot \sigma^{\top})(x)
=2\,\sqrt{g(x)}\,B^{1/2}\,B^{1/2}\,\nabla\sqrt{g(x)}
= B\,\nabla g(x)
=\nabla\!\cdot \mathbf{D}(x),
\]
and the structural identity~\eqref{identity-key} holds pointwise.
One might also use a non-symmetric matrix $\tilde\sigma(x):=\sigma(x)\,Q$ with $Q\in\mathbb{O}(d)$ constant, and the identity remains valid.

\begin{remark}
As a specific example, we consider radial-modulated diffusion. Radial modulation is ubiquitous in rotationally symmetric models, for instance around point sources, and its assessment follows immediately from Case~6. To see this, let $B$ be constant, symmetric, and positive definite as in that case. Let $g(x)=h(r)$ and $r=|x|$, where
$h\in C_b^5([0,\infty))$, $h(r)\ge \alpha>0$, and
\[
h'(0)=h^{(3)}(0)=h^{(5)}(0)=0.
\]
On $\mathbb{R}^d\setminus\{0\}$,
\[
\nabla g(x)=\frac{h'(r)}{r}\,x,
\qquad
\nabla\!\cdot D(x)=B\,\nabla g(x)=\frac{h'(r)}{r}\,B\,x.
\]
The null-odd-derivative conditions ensure that $g$ is $C^5$ at the origin, in particular, $\nabla g(0)=0$, and that all derivatives of $g$ up to order five extend continuously to $\mathbb{R}^d$. For $\sigma(x)=\sqrt{g(x)}\,B^{1/2}$, the chain rule implies that the structural identity \eqref{identity-key} holds for $x\neq 0$, and by continuity it also holds at $x=0$. This case would arise whenever diffusivity depended only on the distance to the origin, for example, in a material whose properties depended on temperature and in the presence of a localized heat source. Note also that this remark covers anisotropic situations, given the generality of $B$.
\end{remark}

\subsection{Negative examples}\label{negative-cases}

In the following, we summarize a couple of cases in which the identity $\nabla\cdot\mathbf{D}=2\,\sigma\,\nabla\cdot\sigma^\top$ does not hold. As already mentioned, this is the generic situation, even in dimension two.

\subsubsection*{\textbf{Case~1: Bidimensional cross-coupled diffusion tensor}}  
Let $\alpha, \,\beta>0$ be constants and let $\tau\in C_b^{5}(\mathbb{R}^2)$ satisfy $\sup_x|\tau(x)|\le \tau_{\max}<\sqrt{\alpha\beta}$. Consider
\[
\mathbf{D}(x)=
\begin{pmatrix}
\alpha^{2}+\tau(x)^{2} & (\alpha+\beta)\,\tau(x)\\[2pt]
(\alpha+\beta)\,\tau(x) & \beta^{2}+\tau(x)^{2}
\end{pmatrix},
\]
for $x\in\mathbb{R}^2$. Since $\tau\in C_b^{5}$, all spatial derivatives of $\mathbf{D}$ up to order five are continuous and bounded, hence~(H1) holds. Also, let
\[
\sigma(x)=
\begin{pmatrix}
\alpha & \tau(x)\\[2pt]
\tau(x) & \beta
\end{pmatrix}
\]
denote its principal square root, $\mathbf D=\sigma\sigma^{\top}$. By the Rayleigh-Ritz theorem, for every $\xi\in\mathbb{R}^2$ we have
\[
\xi^{\top}\mathbf D(x)\xi=|\sigma \xi|^{2}\ \ge\ \lambda_{\min}(\sigma)^{2}\,|\xi|^{2}, 
\]
where $\lambda_{\min}(\sigma)$ denotes the smallest eigenvalue of $\sigma$. Since $\sigma$ is a symmetric $2\times2$ matrix, it follows that $\lambda_{\min}(\sigma)\ge \det\sigma/\operatorname{tr}\sigma$.
Here $\det\sigma=\alpha\beta-\tau(x)^{2}\ge \alpha\beta-\tau_{\max}^{2}=:\delta_1>0$ and
$\operatorname{tr}\sigma=\alpha+\beta$; hence
\begin{equation*}
  \xi^{\top}\mathbf D(x)\xi\ \ge\ \Big(\frac{\delta_1}{\alpha+\beta}\Big)^{2} \;|\xi|^{2}.
\end{equation*}
 Therefore, $\mathbf{D}$ is positive definite for every $x$ and the uniform ellipticity hypothesis~(H2) holds with constant $[\delta_1/(\alpha+\beta)]^{2}$. Subsequently, since in $\mathbf{D}$ only $\tau$ depends on $x$, we obtain
\[
\nabla\!\cdot \mathbf{D}=
\begin{pmatrix}
2\,\tau\,\partial_{1}\tau +(\alpha+\beta)\,\partial_{2}\tau\\[4pt]
(\alpha+\beta)\,\partial_{1}\tau +2\,\tau\,\partial_{2}\tau
\end{pmatrix},
\qquad
\nabla\!\cdot \sigma^{\top}=
\begin{pmatrix}
\partial_{2}\tau\\[4pt]
\partial_{1}\tau
\end{pmatrix}.
\]
Hence,
\[
\Lambda(x):=\nabla\!\cdot \mathbf{D}-2\,\sigma\,(\nabla\!\cdot \sigma^{\top})
=
\begin{pmatrix}
(\beta-\alpha)\,\partial_{2}\tau\\[4pt]
(\alpha-\beta)\,\partial_{1}\tau
\end{pmatrix}.
\]
Thus, if $\alpha\neq\beta$ and $\nabla\tau$ is not identically zero, there exists $x^\ast$ with $\Lambda(x^\ast)\neq 0$; therefore, the structural identity~\eqref{identity-key} is not fulfilled, despite $\mathbf{D}$ satisfies~(H1)--(H2) and $\sigma=\mathbf{D}^{1/2}$ has been chosen as its principal square root.

\begin{remark}
The function $\tau(x)$ controls the off–diagonal part of $\mathbf{D}$ and $\sigma$. Therefore, intuitively, it measures how much the diffusive fluxes correlate between the $x_1$ and $x_2$ directions. If $\tau$ is coordinate dependent, that means the correlation also varies in space. Some admissible forms of $\tau$ that can be expressed in terms of elementary functions are: 
\begin{itemize}
  \item Periodic: $\tau(x)=\varepsilon\,\sin(k^{\top} x)$ with $k\in\mathbb{R}^{2}\setminus\{0\}$, $\varepsilon\in\mathbb{R}$;
  \item Gaussian: $\tau(x)=\varepsilon\,\exp\!\big(-|x-x_{0}|^{2}/\ell^{2}\big)$ with $x_{0}\in\mathbb{R}^{2}$, $\ell>0$, $\varepsilon\in\mathbb{R}$;
  \item Front-like: $\tau(x)=\varepsilon\,\tanh\!\big((x_{1}-c)/\ell\big)$ with $c\in\mathbb{R}$, $\ell>0$, $\varepsilon\in\mathbb{R}$,
  and  $x=(x_1, x_2)$.
\end{itemize}
All these examples give rise to regions where $\nabla\tau(x)$ does not cancel out. There, the function $\Lambda(x)$ does not vanish and the identity~\eqref{identity-key} fails as shown in Case~1.
\end{remark}

\subsubsection*{\textbf{Case~2: Separable cross-coupled diffusion  tensor}} 
Let $\varepsilon\neq 0$ be a constant, and let $a,b\in C_b^5(\mathbb{R})$.
Assume
\[
a(x_1)\ge a_0>0,\qquad b(x_2)\ge b_0>0,\qquad
a(x_1)b(x_2)-\varepsilon^2\ge \delta_2>0\qquad \text{for all }x=(x_1,x_2)\in\mathbb{R}^2.
\]
Let us consider
\[
\mathbf D(x)=
\begin{pmatrix}
a(x_1)^2+\varepsilon^2 & \varepsilon\big(a(x_1)+b(x_2)\big)\\[2pt]
\varepsilon\big(a(x_1)+b(x_2)\big) & b(x_2)^2+\varepsilon^2
\end{pmatrix}.
\]
Since $a,b\in C_b^5$, all spatial derivatives of $\mathbf D$ up to order five are continuous and bounded,
hence~(H1) holds. Moreover,
\[
\det \mathbf D(x)=\big(a(x_1)b(x_2)-\varepsilon^2\big)^2\ge \delta^2_2>0,
\qquad
\mathbf D_{11}(x)=a(x_1)^2+\varepsilon^2>0 .
\]
Therefore, by the Sylvester criterion, $\mathbf D(x)$ is positive definite for every $x$, and~(H2) holds. Also, let us denote by
\[
\sigma(x)=
\begin{pmatrix}
a(x_1) & \varepsilon\\[2pt]
\varepsilon & b(x_2)
\end{pmatrix}
\]
its principal square root; then $\mathbf D=\sigma\sigma^{\top}$ and $\det\sigma(x)=a(x_1)b(x_2)-\varepsilon^2\ge\delta_2>0$.  Subsequently, by separability (that is, in $\mathbf D$ only $a$ depends on $x_1$ and only $b$ on $x_2$) we obtain
\[
\nabla\!\cdot \mathbf D=
\begin{pmatrix}
2a\,\partial_{1}a+\varepsilon\,\partial_{2}b\\[4pt]
\varepsilon\,\partial_{1}a+2b\,\partial_{2}b
\end{pmatrix},
\qquad
\nabla\!\cdot \sigma^{\top}=
\begin{pmatrix}\partial_{1}a\\[2pt]\partial_{2}b\end{pmatrix}.
\]
Hence,
\[
\Lambda(x):=\nabla\!\cdot \mathbf D-2\,\sigma(\nabla\!\cdot \sigma^{\top})
=
-\,\varepsilon
\begin{pmatrix}
\partial_{2}b(x_2)\\[2pt]
\partial_{1}a(x_1)
\end{pmatrix}.
\]
Thus, if $\varepsilon\neq 0$ and either $\partial_{1}a$ or $\partial_{2}b$ is not identically zero, there exists $x^\ast$ with $\Lambda(x^\ast)\neq 0$; therefore, the structural identity~\eqref{identity-key} is not fulfilled, despite $\mathbf D$ again satisfies~(H1)–(H2) and $\sigma=\mathbf D^{1/2}$ is its principal square root.

\begin{remark}
The functions $a(x_1)$ and $b(x_2)$ regulate this cross-coupled diffusion in the directions $x_1$ and $x_2$ respectively. Assume $\varepsilon\neq0$; then the structural identity~\eqref{identity-key} fails precisely in those regions in which either $\partial_{1}a(x_1)$ or $\partial_{2}b(x_2)$ does not vanish. Some admissible choices for $a$ and $b$ that can be written in terms of elementary functions are:
\begin{itemize}
  \item Periodic: $a(x_1)=a_0+\eta_a\sin(k_a x_1+\theta_a)$, \;
  $b(x_2)=b_0+\eta_b\cos(k_b x_2+\theta_b)$, \;
  with $k_a,k_b>0$, $\theta_a,\theta_b\in\mathbb{R}$, and amplitudes chosen so that
  $|\eta_a|<a_0$, $|\eta_b|<b_0$, and $(a_0-|\eta_a|)(b_0-|\eta_b|)-\varepsilon^2\ge \delta_2$.
  \item Gaussian: $a(x_1)=a_0+\eta_a\,e^{-(x_1-x_1^0)^2/\ell_1^2}$, \;
  $b(x_2)=b_0+\eta_b\,e^{-(x_2-x_2^0)^2/\ell_2^2}$, \;
  with $x_1^0,x_2^0\in\mathbb{R}$, $\ell_1,\ell_2>0$, and amplitudes satisfying
  $|\eta_a|<a_0$, $|\eta_b|<b_0$, and $(a_0-|\eta_a|)(b_0-|\eta_b|)-\varepsilon^2\ge \delta_2$.
  \item Front-like: $a(x_1)=a_0+\eta_a\tanh\!\big((x_1-c_1)/\ell_1\big)$, \;
  $b(x_2)=b_0+\eta_b\tanh\!\big((x_2-c_2)/\ell_2\big)$, \;
  with $c_1,c_2\in\mathbb{R}$, $\ell_1,\ell_2>0$, and amplitudes chosen so that
  $|\eta_a|<a_0$, $|\eta_b|<b_0$, and $(a_0-|\eta_a|)(b_0-|\eta_b|)-\varepsilon^2\ge \delta_2$.
\end{itemize}
\end{remark}

\section{Beyond the kinetic interpretation of noise}
\label{sec:beyond}

Our developments so far have left a number of side questions open. The goal of this section is to address them. We start with an attempt of definition of stochastic integral which was aimed to constitute a direct discretization of the kinetic interpretation of noise.

\subsection{The H\"utter-\"Ottinger discretization} \label{subsecho}

In reference~\cite{hutter1998}, the authors propose a new discretization of the stochastic integral that would presumably give rise to the desired calculus. Instead of focusing on the general definition, we will examine a particular case that can be considered paradigmatic:
$$
\int_0^T W_t \; \raisebox{0.5ex}{\scalebox{1.2}{\text{?`}}}\hspace{-0.025cm}W_t := L^2(\Omega)-\lim_{||\Delta_n||\to 0}\sum_{j=1}^{n}
\frac12 \left( \frac{W_{t_j}^2}{W_{t_{j-1}}} +W_{t_{j-1}}\right) \left( W_{t_j}-W_{t_{j-1}}\right),
$$
where $\Delta_n=\{t_0, t_1, \ldots, t_{n-1}, t_n\}$ is a partition of $[0, T]$, $||\Delta_n||=\max_{1\leq k \leq n}(t_k-t_{k-1})$, and ``$\raisebox{0.5ex}{\scalebox{1.2}{\text{?`}}}\hspace{-0.025cm}W_t$'' is our notation for the new attempt of stochastic integration as described in~\cite{hutter1998}. This discretization presents the obvious drawback of being nonlinear for a linear integrand; but, moreover, it is almost surely ill-posed due to the fact that $W_{t_j} - W_{t_{j-1}} \sim \mathcal{N}(0,t_j-t_{j-1})$ along with $W_0=0$ with probability one (thus causing a divergence). In addition, in the interval $[0,T]$, for any $T>0$, the Hausdorff dimension of the set of all zeros of Brownian motion is $1/2$ almost surely, as follows from Theorem~4.24 in~\cite{mp2010bm} and the scale invariance of Brownian motion. In particular, the Brownian motion does not only vanish at the initial time $t=0$, but also in an uncountable set of times within each finite interval starting from it. Just for comparison, consider the Riemann-Stieltjes-type discretization:
\begin{equation}\nonumber
\int_0^T W_t \stackrel{\lambda}{\odot} \mathrm{d} W_t := L^2(\Omega)-\lim_{||\Delta_n||\to 0}\sum_{j=1}^{n} W_{t_{j-1}^*} \left( W_{t_j}-W_{t_{j-1}}\right),
\end{equation}
where $t_{j-1}^* \in [t_{j-1},t_j]$, and therefore it can be expressed as $t_{j-1}^* = \lambda t_{j} + (1-\lambda) t_{j-1}$ for any $\lambda \in [0,1]$. It is well-known that this is a well-defined stochastic integral for any $\lambda$ within this interval and yields the result:
\begin{equation}\nonumber
\int_0^T W_t \stackrel{\lambda}{\odot} \mathrm{d} W_t = \frac{W_T^2}{2} + \left( \lambda -\frac12 \right) T,
\end{equation}
see Chapter~4 in~\cite{evans} for the derivation. So, in contrast with the first discretization, the second choice (which is, in fact, a one-parameter family of discretizations) preserves the linearity of the integrand and gives rise to a well-posed integral.

This attempt of stochastic integral was inspired by a numerical algorithm~\cite{fixman1978}. Although our result in this section is negative, that does not mean that an algorithm cannot be reinterpreted as a stochastic integration scheme. Showing this is the goal of the next subsection.

\subsection{The Fehlberg 255/512-interpretation}

In a sense, as mentioned in the previous subsection, one can think of noise interpretations as suitable limits of some numerical algorithms. In fact, the classical interpretations may be seen as certain limits of Runge-Kutta methods~\cite{hairer}. Precisely, the It\^o interpretation would correspond to the forward Euler method, the HK interpretation to the backward Euler method, and the Stratonovich interpretation to the midpoint method. The last method differs in general to the Heun or trapezoidal method, but they coincide in the limit provided enough regularity is granted~\cite{protter2005stochastic}; that can also be deduced, in a particular case, from the Riemann-Stieltjes-type discretization in the previous subsection. Moreover, most Runge-Kutta methods would reduce to one of these three integration schemes (with a clear preference towards the Stratonovich interpretation) in a suitable limit and if enough regularity is assumed. Interestingly, this is not the case for the first-order Runge-Kutta method introduced by Fehlberg in Section~IV of~\cite{fehlberg}. Assuming enough regularity, this scheme would correspond to the choice $\lambda=255/512$ in the Riemann-Stieltjes scheme introduced in Subsection~\ref{subsecho}; obviously, the interpretations of It\^o, Stratonovich, and HK correspond, respectively, to $\lambda=0,1/2,1$. Since $255/512=1/2-1/512$, this interpretation could be considered as a perturbation of that of Stratonovich. We can introduce it in more generality as follows (from now on we assume $0 \le a < b < \infty$).

\begin{definition}\label{fint}
For any $\Phi\in \mathcal{C}^1(\mathbb{R},\mathbb{R})$ the Fehlberg $255/512-$integral of $(\Phi(W_t))_{t\geq 0}$ in the interval $[a,b]$ with respect to the Brownian motion $\{W_t\}_{t\geq 0}$ is defined as
	\begin{equation}\label{HLI_cs}
	\int\limits_a^b \Phi(W_t)\star \mathrm{d}W_t :=\lim_{\|\Delta_n\|\to 0}\sum_{j=1}^{n}\Phi(W_{t^*_j})(W_{t_j}-W_{t_{j-1}})\quad \textrm{in probability},
	\end{equation}
where $\Delta_n=\{t_0, t_1, \ldots, t_{n-1}, t_n\}$ is a partition of $[a, b]$, $\|\Delta_n\|:=\max_{1\leq j \leq n}(t_j-t_{j-1})$ is its diameter, and $t_{j-1}^* = \lambda t_{j} + (1-\lambda) t_{j-1}$ with $\lambda=255/512$.
\end{definition}

As happens with the Stratonovich and HK integrals, Definition~\ref{fint} poses a limit that exists and is connected to the Itô integral.

\begin{theorem}\label{simtheor}
For any $\Phi \in \mathcal{C}^1(\mathbb{R},\mathbb{R})$ the limit~\eqref{HLI_cs} exists and we have that the equality
\begin{equation}\label{eq:ibf}
\int\limits_a^b \Phi(W_t)\star \,\mathrm{d}W_t=\int\limits_a^b \Phi(W_t)\,\mathrm{d}W_t+ \frac{255}{512}\int\limits_{a}^{b}\Phi^{\prime}(W_t)\,\mathrm{d}t \quad \textrm{holds almost surely}.
\end{equation}
\end{theorem}

\begin{proof}
Following~\cite{escudero2023ito}, the right-hand side of equation~\eqref{eq:ibf} is
almost surely well defined.
Because $\Phi\in \mathcal{C}^1(\mathbb{R},\mathbb{R})$, the It\^o integral is defined as the limit in probability
\begin{equation}\label{II}
\int\limits_a^b \Phi(W_t) \,\mathrm{d}W_t =\lim_{\|\Delta_n\|\to 0}\sum_{j=1}^{n}\Phi(W_{t_{j-1}})(W_{t_j}-W_{t_{j-1}})
\end{equation}
for any family of partitions $\Delta_n$ such that
$\|\Delta_n\|\to 0$ as $n\to \infty$. For one such family $\Delta_n$ take the difference between the limits on the right-hand sides of~\eqref{HLI_cs}
and~\eqref{II} to get
\begin{equation}\label{eq:diferencia}
\begin{split}
D_{\Delta_n} &= \sum_{j=1}^{n}\Big[\Phi(W_{t_j^*})-\Phi(W_{t_{j-1}})\Big](W_{t_j}-W_{t_{j-1}}) \\
&= \sum_{j=1}^{n}\Phi^{\prime}((1-\theta_j)W_{t_{j-1}}+\theta_j W_{t_j^*})(W_{t_j^*}-W_{t_{j-1}})^2 \\
& \quad + \sum_{j=1}^{n}\Phi^{\prime}((1-\theta_j)W_{t_{j-1}}+\theta_j W_{t_j^*})(W_{t_j}-W_{t_{j}^*})(W_{t_j^*}-W_{t_{j-1}}),
\end{split}
\end{equation}
where we have used the continuous differentiability of $\Phi$ together with the mean value theorem; $0 < \theta_j < 1$, $j=1,\ldots,n$ are fixed parameters. Now, by Lemmata~7.2.1 and~7.2.3 in \cite{kuo}, we find
\begin{equation}\label{eq:limitderi}
\lim_{\|\Delta_n\|\to 0}\sum_{j=1}^{n}\Phi^{\prime}\big((1-\theta_j)W_{t_{j-1}}+\theta_j W_{t_j}\big)(W_{t_j^*}-W_{t_{j-1}})^2= \lambda \int\limits_{a}^{b}\Phi^{\prime}(W_t)\,\mathrm{d}t
\end{equation}
in probability.

To analyze the last line of equation~\eqref{eq:diferencia}, note that if $\Phi'$ were a bounded function (additionally to continuous), then
\begin{equation}\label{eq:crossterm}
\begin{split}
& \qquad \mathbb{E} \left[ \left|\sum_{j=1}^{n}\Phi^{\prime}((1-\theta_j)W_{t_{j-1}}+\theta_j W_{t_j^*})(W_{t_j}-W_{t_{j}^*})(W_{t_j^*}-W_{t_{j-1}}) \right|\right] \\
&\le \mathbb{E} \left[\sum_{j=1}^{n} \left|\Phi^{\prime}((1-\theta_j)W_{t_{j-1}}+\theta_j W_{t_j^*})\right| \left|W_{t_j^*}-W_{t_{j-1}}\right| \mathbb{E} \left[\left|W_{t_j}-W_{t_{j}^*}\right| \Big| \mathcal{F}_{t_j^*} \right] \right] \\
&\le \sqrt{\frac{2 (1-\lambda) \|\Delta_n\|}{\pi}} \; \mathbb{E} \left[\sum_{j=1}^{n} \left|\Phi^{\prime}((1-\theta_j)W_{t_{j-1}}+\theta_j W_{t_j^*})\right| \left|W_{t_j^*}-W_{t_{j-1}}\right| \right] \\
&\le \sqrt{\frac{2 (1-\lambda) \|\Delta_n\|}{\pi}} \, \|\Phi'\|_\infty \; \mathbb{E} \left[\sum_{j=1}^{n} \left|W_{t_j^*}-W_{t_{j-1}}\right|^2 \right]^{1/2} \\
&= \sqrt{\frac{2 \lambda (1-\lambda) \|\Delta_n\|}{\pi}} \, \|\Phi'\|_\infty \, \left[\sum_{j=1}^{n} (t_j - t_{j-1}) \right]^{1/2} \\
&= \sqrt{\frac{2 \lambda (1-\lambda) \|\Delta_n\|}{\pi}} \, \|\Phi'\|_\infty \sqrt{b-a} \longrightarrow 0, \quad \text{as} \quad n \to \infty,
\end{split}
\end{equation}
by the triangle inequality, the tower property, the independence of the increments of Brownian motion, the value of the first absolute moment of Brownian motion, the Jensen inequality, and the variance of the Brownian increments. Thus, the last line of equation~\eqref{eq:diferencia} converges to zero in $L^1(\Omega)$, and hence in probability (provided $\|\Phi'\|_\infty$ is well defined).

However, $\Phi'$ is continuous but not globally bounded in general; nevertheless, its continuity implies it is bounded on any closed interval, so the maximum
$$
\max_{W_t \in [-l,l]}\Phi^{\prime}((1-\theta_j)W_{t_{j-1}}+\theta_j W_{t_j^*})
$$
is well defined, where the full trajectory of the Brownian motion, for $a \le t \le b$, is assumed to be constrained in $[-l,l]$, with $l>0$ large enough. In consequence, by the almost sure continuity of Brownian motion and~\eqref{eq:crossterm},
\begin{equation}\nonumber
\lim_{n \to \infty} \mathbb{P} \left[ \left. \left|\sum_{j=1}^{n}\Phi^{\prime}((1-\theta_j)W_{t_{j-1}}+\theta_j W_{t_j^*})(W_{t_j}-W_{t_{j}^*})(W_{t_j^*}-W_{t_{j-1}}) \right| > \epsilon \; \right| \max_{a \le t \le b} |W_t| \le l \right]=0,
\end{equation}
for all $\epsilon > 0$. Now, by the law of total probability,
\begin{equation}\label{eq:crossterm2}
\begin{split}
& \qquad \mathbb{P} \left[ \left|\sum_{j=1}^{n}\Phi^{\prime}((1-\theta_j)W_{t_{j-1}}+\theta_j W_{t_j^*})(W_{t_j}-W_{t_{j}^*})(W_{t_j^*}-W_{t_{j-1}}) \right| > \epsilon \right] \\
&= \mathbb{P} \left[ \left. \left|\sum_{j=1}^{n}\Phi^{\prime}((1-\theta_j)W_{t_{j-1}}+\theta_j W_{t_j^*})(W_{t_j}-W_{t_{j}^*})(W_{t_j^*}-W_{t_{j-1}}) \right| > \epsilon \; \right| \max_{a \le t \le b} |W_t| \le l \right] \\
& \quad \times \mathbb{P} \left[ \max_{a \le t \le b} |W_t| \le l \right] \\
& \quad + \mathbb{P} \left[ \left. \left|\sum_{j=1}^{n}\Phi^{\prime}((1-\theta_j)W_{t_{j-1}}+\theta_j W_{t_j^*})(W_{t_j}-W_{t_{j}^*})(W_{t_j^*}-W_{t_{j-1}}) \right| > \epsilon \; \right| \max_{a \le t \le b} |W_t| \not\le l \right] \\
& \quad \times \mathbb{P} \left[ \max_{a \le t \le b} |W_t| \not\le l \right] \\
&\le \mathbb{P} \left[ \left. \left|\sum_{j=1}^{n}\Phi^{\prime}((1-\theta_j)W_{t_{j-1}}+\theta_j W_{t_j^*})(W_{t_j}-W_{t_{j}^*})(W_{t_j^*}-W_{t_{j-1}}) \right| > \epsilon \; \right| \max_{a \le t \le b} |W_t| \le l \right] \\
& \quad + \mathbb{P} \left[ \max_{a \le t \le b} |W_t| \not\le l \right] \\
&\le \mathbb{P} \left[ \left. \left|\sum_{j=1}^{n}\Phi^{\prime}((1-\theta_j)W_{t_{j-1}}+\theta_j W_{t_j^*})(W_{t_j}-W_{t_{j}^*})(W_{t_j^*}-W_{t_{j-1}}) \right| > \epsilon \; \right| \max_{a \le t \le b} |W_t| \le l \right] \\
& \quad + \frac{1}{l} \; \mathbb{E} \left( |W_b| \right) \\
&= \mathbb{P} \left[ \left. \left|\sum_{j=1}^{n}\Phi^{\prime}((1-\theta_j)W_{t_{j-1}}+\theta_j W_{t_j^*})(W_{t_j}-W_{t_{j}^*})(W_{t_j^*}-W_{t_{j-1}}) \right| > \epsilon \; \right| \max_{a \le t \le b} |W_t| \le l \right] \\
& \quad + \frac{1}{l} \, \sqrt{\frac{2 b}{\pi}}
\longrightarrow 0, \quad \text{as} \quad l, n \to \infty,
\end{split}
\end{equation}
where we have used the Doob submartingale inequality in the previous-to-last step, see Theorem~4.5.1 in~\cite{kuo}, and the value of the first absolute moment of Brownian motion in the final one.

Expressions~\eqref{II}, \eqref{eq:diferencia}, \eqref{eq:limitderi}, and~\eq{eq:crossterm2} prove the existence of the left-hand side of equation~\eqref{HLI_cs}. Lastly, by taking a subsequence if necessary, we deduce
\begin{equation*}
\int\limits_a^b \Phi(W_t)\star \,\mathrm{d}W_t=\int\limits_a^b \Phi(W_t)\,\mathrm{d}W_t+ \frac{255}{512} \int\limits_{a}^{b}\Phi^{\prime}(W_t)\,\mathrm{d}t \quad \textrm{almost surely},
\end{equation*}
since $\lambda=255/512$ by hypothesis.
\end{proof}

Now, we will illustrate the use of these developments in an example. Let us assume that a given physical object moves with a velocity that can be assimilated to a constant multiple of Brownian motion; the constant will henceforth be denoted as $k$. If its mass were unitary, what can be assumed without loss of generality, its kinetic energy $Q_t=k^2 W_t^2/2$ would fulfill:
$$
\mathrm{d}Q_t= k^2 \, W_t \, \mathrm{d} W_t + \frac{k^2}{2} \, \mathrm{d}t,
$$
as a consequence of the It\^o lemma;
so it obeys the stochastic differential equation
$$
\mathrm{d}Q_t= \frac{k^2}{2} \, \mathrm{d}t + k \, \sqrt{2 Q_t} \, \mathrm{d} W_t.
$$
We will suppose that the particle is initially at rest, so the initial condition is $Q_0=0$.
This equation possesses a unique solution, as guaranteed by the Watanabe-Yamada theorem, see Theorem~1 in~\cite{wy}. However, if we employed Stratonovich calculus, we would find
$$
\mathrm{d}Q_t= k^2 \, W_t \circ \mathrm{d} W_t.
$$
This equation is solved by $Q_t=k^2 W_t^2/2$, by $Q_t=0$, and by the uncountable family of solutions
$$
Q_t= \frac{k^2}{2} \, \left(\overleftrightarrow{W}_{t-\tau}\right)^2,
$$
where $\overleftrightarrow{W}_{t}:=W_t$ if $t>0$ and $\overleftrightarrow{W}_{t}:=0$ if $t \le 0$, with $\tau$ being any almost surely positive random variable independent of $W_t$. Note that, unlike $Q_t=k^2 W_t^2/2$ and $Q_t=0$ that are strong solutions, all solutions in this family are weak, see Section~5.3 in~\cite{oksendal}. We emphasize that this is not a problem since, first, weak solutions are those that matter in physics (where it is common to focus on the probability density function of the solution rather than on its paths), and also because infinitely many strong solutions can be built for this model following the procedures in~\cite{cescudero} and~\cite{cescudero2}. The reason for this infinite multiplicity of solutions is the transformation of the origin from being an instantaneously reflecting boundary to being an absorbing one as we shift the interpretation from It\^o to Stratonovich, see~\cite{ce}. On the other hand, we could find the stochastic differential equation in Fehlberg form
$$
\mathrm{d}Q_t= \frac{k^2}{1024} \, \mathrm{d}t + k \, \sqrt{2 Q_t} \star \mathrm{d} W_t,
$$
which is a direct consequence of Theorem~\ref{simtheor}. This equation approximates in a sense the Stratonovich interpretation, but preserves the reflecting character of the boundary at zero, so it is free from the infinite multiplicity of solutions. Of course, the same result follows from any other $\lambda=1/2-\epsilon$ with $0<\epsilon<1/2$. Even one can take the limit $\epsilon \searrow 0$ to recover the unique Stratonovich solution that owns physical meaning. The advantage of the choice $\lambda=255/512$ is simply the preexistence of the Runge-Kutta rule introduced by Fehlberg in~\cite{fehlberg}. So, in a sense, the choice $\epsilon=1/512$ could be considered as a rule of thumb to select a noise interpretation close to that of Stratonovich but free from the multiplicity of solutions generated by a square-root diffusion term (a type of diffusion term that is not rare in physical modeling~\cite{ce}).

As a final note to this subsection, let us mention that modeling the velocity of a physical object with a Brownian motion is not as common as describing it like an Ornstein-Uhlenbeck process; the latter choice is known as the Langevin model for the random dispersal of a particle. However, two things should be said in this regard. First, the Brownian motion approximates the Langevin dispersal in the limit of vanishing friction, so our Brownian model can be considered as an effective description of the object velocity in such a limit. Moreover, all the developments presented here for a particle moving with a Brownian speed can be immediately translated to the case of a Langevin particle. To this end, one just needs to repeat the arguments set forth in~\cite{cescudero,cescudero2,escudero2023ito} {\it mutatis mutandis}.

The next subsection shows some possible consequences of removing one of our key assumptions: uniform ellipticity. This will be illustrated with the three classical interpretations of noise together with the one just introduced in this subsection.

\subsection{Application to heterogeneous diffusions in the absence of uniform ellipticity}\label{sec:heterogeneous}

In this subsection we follow reference~\cite{ccm2013anomalous}, where the following model (except for the different notation) for diffusion in heterogeneous media is studied:
$$
\frac{\mathrm{d}X_t}{\mathrm{d}t}= k \, |X_t|^\alpha \, \xi(t),
$$
where $k>0$, $\alpha \in \mathbb{R}$, $\xi(t)$ is white noise, and the interpretation is that of Stratonovich. Clearly, its diffusion term is degenerated. While the physical consequences of such a model were investigated in~\cite{ccm2013anomalous}, herein we will focus on its stochastic analytical properties. Our arguments will be akin to those employed in~\cite{ce,cescudero,cescudero2,escudero2023ito}.

As a first step, we translate the model into the precise stochastic differential equation
$$
\mathrm{d}X_t= \frac{\alpha k^2}{2} \, X_t^{2 \alpha -1} \, \mathrm{d}t + k \, X_t^\alpha \, \mathrm{d}W_t,
$$
which should be a reasonable proposal while the position of the particle that diffuses in the heterogeneous medium, $X_t$, stays nonnegative (since the absolute value in the original model has been removed). Of course, this means that we need to assume an initial condition $X_0=x_0 \ge 0$. It turns out that this stochastic differential equation presents a rich phenomenology in terms of its solvability as the parameter $\alpha$ is varied. This is shown with particular cases in what follows. They illustrate how the lack of uniform ellipticity can be more or less problematic, depending on the specific value of $\alpha$.

We start with the value $\alpha=1$, for which the equation possesses a unique solution that is strong and global, and given by the explicit formula
$$
X_t = x_0 \, \exp \left(k \, W_t \right).
$$
This is a consequence of the classical existence theory (see, for instance, Theorem~10.3.5 in~\cite{kuo} and Theorem 5.2.1 in~\cite{oksendal}). Moreover, the solution remains positive (null) for all times if the initial condition is positive (null), so removing the absolute value is appropriate in this case. Also, the Stratonovic formulation $\mathrm{d}X_t= k \, X_t \circ \mathrm{d}W_t$ is perfectly valid in this case too.

If $\alpha=2$, we can use It\^o calculus to find
$$
X_t=\frac{1}{x_0^{-1}-k\, W_t},
$$
whenever $x_0>0$ and $X_t=0$ for $x_0=0$. So in the latter case we have global-in-time existence of the solution, while in the former it blows up at the stopping time
$$
\tau_2=\inf\{t>0 \, | \, W_t=1/(x_0 \, k)\},
$$
which is almost surely positive and almost surely finite. In other words, the solution ceases to exist at a random time that is finite with probability one; but in the mean time it is positive, so the removal of the absolute value from the original model is legitimate. Additionally, during the lapse of existence of the solution, the equation might be converted to Stratonovich form without altering any of these results.

For $\alpha=1/2$, the model reads
$$
\mathrm{d}X_t= \frac{k^2}{4} \, \mathrm{d}t + k \, \sqrt{X_t} \, \mathrm{d}W_t.
$$
The change $X_t=Q_t/2$ reduces this model to exactly that of the previous subsection, so all said there about It\^o, Stratonovich, and Fehlberg applies here identically whenever $x_0=0$. For $x_0>0$, we have the strong solution
\begin{equation}\nonumber
X_t= \left(\frac{k}{2} \,W_t + \sqrt{x_0}\right)^2,
\end{equation}
which is nonnegative for all times (thus justifying the removal of the absolute value in the equation) and positive for $0 \le t < \tau_{1/2}$, where
$$
\tau_{1/2} =\inf\{t>0 \, | \, W_t=-2\sqrt{x_0}/k\},
$$
which is an almost surely finite and positive stopping time. Therefore, the solution gets to zero in finite time almost surely, and the same situation described in the previous subsection is reproduced again after the time lapse $\tau_{1/2}$.

For $\alpha=3/4$ we have
$$
\mathrm{d}X_t= \frac{3 k^2}{8} \, \sqrt{X_t} \, \mathrm{d}t + k \, \sqrt[4]{X_t^3} \, \mathrm{d}W_t
\quad \text{and} \quad \mathrm{d}X_t= k \, \sqrt[4]{X_t^3} \circ \mathrm{d}W_t.
$$
For $x_0=0$ we have as solution $X_t=0$ and for $x_0 \ge 0$ we obtain
\begin{equation}\nonumber
X_t= \left(\frac{k}{4} \,W_t + \sqrt[4]{x_0}\right)^4
\end{equation}
both for It\^o and Stratonovich. So the same arguments of the previous subsection and the previous paragraph can be reproduced in this case to prove that the infinite multiplicity of solutions happens in this case in finite time almost surely for \emph{both} interpretations of noise.

Finally, consider $\alpha=1/4$ to arrive at the model
$$
\mathrm{d}X_t= \frac{k^2}{8} \, \frac{1}{\sqrt{X_t}} \, \mathrm{d}t + k \, \sqrt[4]{X_t} \, \mathrm{d}W_t \quad \text{and} \quad \mathrm{d}X_t= k \, \sqrt[4]{X_t} \circ \mathrm{d}W_t.
$$
While the Stratonovich version can be posed for $x_0 \ge 0$, the It\^o one can only be posed for $x_0>0$; in this latter case the solution reads
\begin{equation}\label{eq:solution34}
X_t= \sqrt[3]{\left(\frac{3k}{4} \,W_t + \sqrt[4]{x_0^3}\right)^4}
\end{equation}
for \emph{both} interpretations. Moreover, the Stratonovich interpretation admits the solution $X_t=0$ for $x_0=0$. Then, arguing as before, the Stratonovich interpretation produces an infinite multiplicity of solutions in finite time almost surely. This can be rectified by imposing an instantaneously reflecting boundary at the origin, which selects one solution out of the infinitely many; precisely, the one given by~\eqref{eq:solution34}. The situation gets worse in the case of It\^o, in which the solution ceases to exist in finite time almost surely due to the divergence in the drift.

All in all, we have shown the enormous consequences that changing the noise interpretation might have regarding existence and uniqueness of solution with a single model provided with a tunable parameter (and in the absence of uniform ellipticity). For some values of this parameter, the outcome is independent of the noise interpretation, while for others It\^o is better, and even Stratonovich can be better for some. One can easily check that, in all the cases analyzed, the Fehlberg interpretation behaves as the It\^o one, while the HK interpretation can only worsen things or leave them the same; this latter claim follows easily from the arguments set forth in~\cite{escudero2023ito}. Also, in all these cases, the removal of the absolute value from the original stochastic differential equation is justified by the nonnegativity of the solutions. As a final note, let us mention that these conclusions affect only the cases characterized by the explicitly examined values of $\alpha$. Due to the complexity of the model, each value of $\alpha$ leads to an equation that deserves its own study. 

One of the advantages of these models is that they are able to produce anomalously diffusing random motions~\cite{ccm2013anomalous}. The next subsection will show that the same is accessible for other models that are free from the mathematical pathologies exposed here.

\subsection{The case of scaled Brownian motion}

The previous subsection focused on heterogeneous diffusion processes, which were introduced in the physical literature as capable of producing anomalously diffusing paths~\cite{ccm2013anomalous}. Yet another process that enjoys the same property is the scaled Brownian motion, which is the solution of
$$
\frac{\mathrm{d}X_t}{\mathrm{d}t}= F(t) \, \xi(t),
$$
where $\xi(t)$ is white noise and $F:[0,T] \longrightarrow \mathbb{R}$ is deterministic, and has also been introduced in the physical literature~\cite{jeon2014scaled}. The use of power laws, which are not necessarily continuously differentiable at the origin, as nonlinearities in the previous subsection created some technical difficulties. This raises the question of possible mathematical difficulties when power laws are used for $F(t)$, as is done in~\cite{jeon2014scaled}. In this subsection we show that, contrary to what happened in the two previous subsections, the change in interpretation does not alter the results at all. Moreover, these are free from pathologies and reduce to the classical results one obtains for deterministic integrals under a mild regularity assumption. Since the equation for the scaled Brownian motion is just a stochastic integral, we start introducing what we mean by this sort of integration.

\begin{definition}\label{defirs}
For any $0 \le a < b < \infty$ and any $F:[a,b] \longrightarrow \mathbb{R}$ continuous and of bounded variation, its $\lambda-$integral with respect to the Brownian motion $\{W_t\}_{t\geq 0}$ is defined as
\begin{equation}\label{IRS}
\int\limits_a^b F(t) \stackrel{\lambda}{\odot} \mathrm{d}W_t :=\lim_{\|\Delta_n\|\to 0}\sum_{j=1}^{n} F(t^*_j)(W_{t_j}-W_{t_{j-1}})\quad \textrm{in probability},
\end{equation}
where $\Delta_n=\{t_0, t_1, \ldots, t_{n-1}, t_n\}$ is a partition of $[a, b]$, $\|\Delta_n\|:=\max_{1\leq j \leq n}(t_j-t_{j-1})$ is its diameter, and $t_{j-1}^* = \lambda t_{j} + (1-\lambda) t_{j-1}$ with $0 \le \lambda \le 1$.
\end{definition}

The main result of this section proves that the limit exists and is uniform in $\lambda$. In other words, the equation is independent of the interpretation. Note that the result is obvious from the established literature for continuously differentiable functions, but here we assume much less regularity. In particular, $F(\cdot)$ is not necessarily absolutely continuous, so its Radon-Nikodym derivative does not need to be well defined.

\begin{lemma}\label{lemirs}
The limit~\eqref{IRS} exists and moreover
\begin{equation}\nonumber
\int\limits_a^b F(t) \stackrel{\lambda}{\odot} \mathrm{d}W_t=\int\limits_a^b F(t)\,\mathrm{d}W_t \quad \textrm{almost surely}.
\end{equation}
\end{lemma}

\begin{proof}
The integral on the right-hand side is a well-defined It\^o integral that admits the Riemann sum approximation in Definition~\ref{defirs} with $\lambda=0$, see Theorem~21 of Section~5 in Chapter~II of~\cite{protter2005stochastic}; hence the case $\lambda=0$ is a tautology. For $\lambda \neq 0$ compute
\begin{equation}\nonumber
\int\limits_a^b F(t) \stackrel{\lambda}{\odot} \mathrm{d}W_t - \int\limits_a^b F(t)\,\mathrm{d}W_t= \mathbb{P}-\lim_{\|\Delta_n\|\to 0} D_{\Delta_n},
\end{equation}
where
\begin{equation}\nonumber
D_{\Delta_n} = \sum_{j=1}^{n}\Big[F(t_j^*)-
F(t_{j-1})\Big](W_{t_j}-W_{t_{j-1}}).
\end{equation}
Therefore
\begin{equation}\nonumber
\begin{split}
\mathbb{E} \left[ \left|D_{\Delta_n} \right| \right] &= \mathbb{E} \left[ \left|\sum_{j=1}^{n}\Big[F(t_j^*)-
F(t_{j-1})\Big](W_{t_j}-W_{t_{j-1}}) \right| \right] \\
&\le \sum_{j=1}^{n} \left| F(t_j^*)-
F(t_{j-1})\right| \, \mathbb{E} \left[ \left| W_{t_j}-W_{t_{j-1}} \right|\right] \\
&\le \sqrt{\frac{2 \|\Delta_n\|}{\pi}} \, \sum_{j=1}^{n} \left| F(t_j^*)-
F(t_{j-1})\right|
\\
&\le \sqrt{\frac{2 \|\Delta_n\|}{\pi}} \, \sum_{j=1}^{n} \left[ \left| F(t_j)-
F(t_{j}^*)\right| + \left| F(t_j^*)-
F(t_{j-1})\right| \right] \longrightarrow 0, \quad \text{as} \quad n \to \infty,
\end{split}
\end{equation}
by the triangle inequality, the value of the first absolute moment of the Brownian increments, and the bounded variation of the integrand. Then,
$$
L^1(\Omega)-\lim_{\|\Delta_n\|\to 0} D_{\Delta_n}=0 \, \Longrightarrow \, \mathbb{P}-\lim_{\|\Delta_n\|\to 0} D_{\Delta_n}=0,
$$
and, hence, the limit~\eqref{IRS} exists. Now, by passing to a subsequence if necessary, the equality in the statement is established almost surely.
\end{proof}

\begin{remark}
Note that the hypothesis on the bounded variation of the integrands covers the case of simple power laws studied in~\cite{jeon2014scaled}. In fact, it covers the case of any function that is monotonic on the integration interval, assumed compact.
\end{remark}

\begin{corollary}
For any $0 \le a < b < \infty$ and any $F:[a,b] \longrightarrow \mathbb{R}$ continuous and of bounded variation, its $\lambda-$integral 
\begin{equation}\nonumber
\int\limits_a^b F(t) \stackrel{\lambda}{\odot} \mathrm{d}W_t = F(b) \, W_b - F(a) \, W_a
-\int_a^b W_t \, \mathrm{d}F(t)
\end{equation}
almost surely for every $0 \le \lambda \le 1$.
\end{corollary}

\begin{proof}
First note that the integral on the right-hand side is almost surely well defined as a Riemann–Stieltjes integral because the integrator is of bounded variation and the integrand is continuous with probability one. Moreover, the integral on the left-hand side is well defined by Lemma~\ref{lemirs}. Again by this lemma, we have that
\begin{equation}\nonumber
\begin{split}
\int\limits_a^b F(t) \stackrel{\lambda}{\odot} \mathrm{d}W_t &=\int\limits_a^b F(t)\,\mathrm{d}W_t \\
&= \mathbb{P}-\lim_{||\Delta_n||\to 0}\sum_{j=1}^{n} F(t_{j-1}) \left( W_{t_j}-W_{t_{j-1}}\right).
\end{split}
\end{equation}
On the other hand, the integral
\begin{equation}\nonumber
\begin{split}
\int_a^b W_t \, \mathrm{d}F(t) &= \lim_{||\Delta_n||\to 0}\sum_{j=1}^{n} W_{t_{j-1}} \left[ F(t_{j}) - F(t_{j-1})\right] \\
&= \lim_{||\Delta_n||\to 0}\sum_{j=1}^{n} W_{t_{j}} \left[ F(t_{j}) - F(t_{j-1})\right]
\end{split}
\end{equation}
almost surely by the properties of Riemann–Stieltjes integration. Now, since almost sure convergence implies convergence in probability, we find
\begin{eqnarray}\nonumber
&& \int\limits_a^b F(t) \stackrel{\lambda}{\odot} \mathrm{d}W_t + \int_a^b W_t \, \mathrm{d}F(t) \\ \nonumber
&=& \mathbb{P}-\lim_{||\Delta_n||\to 0}\sum_{j=1}^{n} F(t_{j-1}) \left( W_{t_j}-W_{t_{j-1}}\right) + \mathbb{P}-\lim_{||\Delta_n||\to 0}\sum_{j=1}^{n} W_{t_{j}} \left[ F(t_{j}) - F(t_{j-1})\right]
\\ \nonumber
&=& F(b) \, W_b - F(a) \, W_a,
\end{eqnarray}
because the cancellation of the mixed terms leads to a telescopic sum. By passing to a subsequence if necessary, the equality is the statement becomes established.
\end{proof}

\begin{remark}
This corollary shows that the classical integration by parts formula of the Riemann–Stieltjes integral is still valid in this case.
\end{remark}

Note that the two results introduced in this subsection extend, for the particular case of a deterministic integrand, Theorem~8.3.4 in~\cite{kuo} and Theorem~30 in Section~5 of Chapter~V in~\cite{protter2005stochastic} (with, obviously, many more simplifying assumptions in the case of this second statement). Note also that they directly imply existence and uniqueness of solution to the scaled Brownian motion equation. Some types of singular functions that can be integrated within the present framework, but not in the classical one, are devil staircases~\cite{falconer}, the Minkowski question mark function~\cite{paradis2001derivative}, the Lebesgue function~\cite{kawamura2011set}, and their extensions~\cite{sanchez2012singular}.

Overall, models for anomalously diffusing particles have been introduced in~\cite{ccm2013anomalous} and~\cite{jeon2014scaled}. While both types of models rely on the presence of power laws, this type of functional dependence is introduced in different manners. This implies that the first sort of models presents technical difficulties that are absent in the second, as we have proven in this and the previous subsection.

\section{Conclusions}
\label{sec:conclusions}

The notion of kinetic interpretation of noise was introduced in~\cite{hutter1998}, although it might have been introduced under another format elsewhere, due to the natural character of this idea. Indeed, it could seem appealing to consider the diffusion tensor as more fundamental than the noise amplitude (which is its square root), given its appearance in the fluctuation-dissipation relation. As such, it can be considered a reasonable datum to provide as input for a mathematical model. If this model were a Fokker-Planck equation, we would have only two possibilities: the Itô form~\eqref{CDE0} and the Fick form~\eqref{opeEquation}. Any other form of the Fokker-Planck equation relies on a fractional power of the diffusion tensor rather than on the tensor itself (this can be seen, for instance, in equation~(9) of~\cite{pachecopozo2025njp}). Since there is a clear connection between the Itô form of the Fokker-Planck equation and SDEs (see Section~\ref{sec:cdequations}), one could be tempted to conjecture some similar relation starred by the Fick form. However, this relation only exists under a very restrictive structural condition on the diffusion tensor (at least, if one wants to remove the extra drift term), precisely $\nabla \cdot \mathbf{D} = 2 \sigma \nabla \cdot \sigma^\top$, what makes it highly non-generic (except in one dimension) due to the non-commutativity of the matrix algebra. Section~\ref{sec:cdehkc} was devoted to state this condition as well as to give examples of diffusion tensors that either fulfill it or not. Still, some of the very structured tensors that fulfill it are of physical relevance.

Section~\ref{sec:beyond} addresses a number of problems that arise from the analysis of the kinetic interpretation of noise. The first is the notion of stochastic integration introduced in~\cite{hutter1998}, which we have seen to be ill-posed. This notion was motivated by the classical \emph{Fixman algorithm}~\cite{fixman1978}. To show that it is possible to introduce notions of stochastic integration from numerical algorithms, we illustrated so with the Fehlberg $255/512-$interpretation in that section. Subsequently, we used it, along with the other three more classical interpretations, to model stochastic transport in heterogeneous media. Moreover, we employed the resulting models to illustrate what happens when one of our main hypotheses, the uniform ellipticity assumption, is removed. Finally, we analyzed a family of models that are capable of producing anomalously diffusing paths, as the previous ones, but, in contrast to those, enjoy much more mathematical tractability.

Although our work is primarily methodological and aims to establish a mathematical framework for the rigorous analysis of physical models, it may also have practical implications. As already mentioned, the attempt of stochastic integration in~\cite{hutter1998} was motivated by the Fixman algorithm~\cite{fixman1978}. More recent numerical schemes depart from this~\cite{delong2014a,delong2013,delong2014}. Our work suggests reconsidering numerical analytical methods that arise from the kinetic interpretation of noise. In particular, it points to a direct numerical approach: determining the Itô SDE corresponding to a given convection-diffusion equation as shown in Section~\ref{sec:cdequations}, and subsequently employing the numerical method in~\cite{chen2018numerical}. However, given the complexity of fluctuating hydrodynamics, it is possible that a different approach is required in that field.

We also believe that our developments may help lay the groundwork for a theoretical framework for designing and analyzing media capable of manipulating diffusive trajectories. Specifically, spatially engineered diffusion tensors, coupled with an appropriate choice of stochastic interpretation, can give rise to the phenomenon of \emph{stochastic cloaking}, in which certain regions of space are effectively hidden from particles undergoing random diffusive motion. Stochastic cloaking thus represents a potential application of the principles of stochastic calculus introduced here to problems involving the control of random transport processes. In this sense, our framework may provide part of the mathematical foundation for the design of metamaterials capable of concealing regions of space from diffusing particles. In~\cite{roberts2024cloaking}, it is shown how the properties of the diffusion tensor, together with the choice of stochastic interpretation, can be combined to controllably modify diffusion in space. In particular, that work suggests that the choice of the Itô interpretation is key to achieving this effect, despite the seemingly more natural character of the Fick law in that context.

Reference~\cite{roberts2024cloaking} provides yet another example, along with~\cite{bhattacharyay2020generalization,bhattacharyay2025active,ce,dhawan2024ito,cescudero,cescudero2,em,escudero2023ito}, of the greater flexibility of the Itô interpretation in a wide range of physical applications. In part, our present work points in the same direction, but, simultaneously, highlights a specific structure of the diffusion tensor that permits giving a precise meaning to the kinetic interpretation of noise. Perhaps that interpretation is not as flexible as that of Itô, but the highlighted structure could be of relevance in some physical problems.

\section*{Acknowledgements}

This work has been partially supported by the Government of Spain (Ministerio de Ciencia, Innovación y Universidades) and the European Union through Projects PID2024-158823NB-I00 and
CPP2024-011557.

\bibliographystyle{plain}  
\bibliography{biblio}

\begin{thebibliography}{10}

\bibitem{ancey2015sads}
C.~Ancey, P.~Bohorquez, and J.~Heyman.
\newblock \textit{Stochastic interpretation of the advection-diffusion equation and its relevance to bed load transport}.
\newblock {\em \textnormal{Journal of Geophysical Research: Earth Surface}}, 120:2529--2551, 2015.

\bibitem{bachelier}
L.~Bachelier.
\newblock \textit{Th{\'e}orie de la sp{\'e}culation}.
\newblock {\em \textnormal{Annales Scientifiques de l'{\'E}cole Normale Sup{\'e}rieure}}, 3:21--86, 1900.

\bibitem{bhatia2009positive}
Rajendra Bhatia.
\newblock {\em \textit{Positive Definite Matrices}}.
\newblock Princeton University Press, Princeton, 2009.

\bibitem{bhattacharyay2020generalization}
A.~Bhattacharyay.
\newblock \textit{Generalization of Stokes-Einstein relation to coordinate dependent damping and diffusivity: an apparent conflict}.
\newblock {\em \textnormal{Journal of Physics A: Mathematical and Theoretical}}, 53:075002, 2020.

\bibitem{bhattacharyay2025active}
A.~Bhattacharyay.
\newblock Brownian motion near a wall: the dilemma of {I}t{\^o} or {S}tratonovich.
\newblock {\em \textnormal{Journal of Physics A: Mathematical and Theoretical}}, 58:213001, 2025.

\bibitem{callen1951}
H.~B. Callen and T.~A. Welton.
\newblock \textit{Irreversibility and generalized noise}.
\newblock {\em \textnormal{Physical Review}}, 83:34--40, 1951.

\bibitem{chen2018numerical}
L.~Chen, E.~R. Jakobsen, and A.~Naess.
\newblock \textit{On numerical density approximations of solutions of SDEs with unbounded coefficients}.
\newblock {\em \textnormal{Advances in Computational Mathematics}}, 44:693--721, 2018.

\bibitem{ccm2013anomalous}
A.~G. Cherstvy, A.~V. Chechkin, and R.~Metzler.
\newblock \textit{Anomalous diffusion and ergodicity breaking in heterogeneous diffusion processes}.
\newblock {\em \textnormal{New Journal of Physics}}, 15:083039, 2013.

\bibitem{ce}
{\'A}.~Correales and C.~Escudero.
\newblock \textit{It{\^o} vs Stratonovich in the presence of absorbing states}.
\newblock {\em \textnormal{Journal of Mathematical Physics}}, 60:123301, 2019.

\bibitem{delong2014a}
S.~Delong, F.~Balboa~Usabiaga, R.~Delgado-Buscalioni, B.~E. Griffith, and A.~Donev.
\newblock \textit{Brownian dynamics without Green's functions}.
\newblock {\em \textnormal{The Journal of Chemical Physics}}, 140:134110, 2014.

\bibitem{delong2013}
S.~Delong, B.~E. Griffith, E.~Vanden-Eijnden, and A.~Donev.
\newblock \textit{Temporal integrators for fluctuating hydrodynamics}.
\newblock {\em \textnormal{Physical Review E}}, 87:033302, 2013.

\bibitem{delong2014}
S.~Delong, Y.~Sun, B.~E. Griffith, E.~Vanden-Eijnden, and A.~Donev.
\newblock \textit{Multiscale temporal integrators for fluctuating hydrodynamics}.
\newblock {\em \textnormal{Physical Review E}}, 90:063312, 2014.

\bibitem{dhawan2024ito}
A.~Dhawan and A.~Bhattacharyay.
\newblock \textit{It{\^o}-distribution from Gibbs measure and a comparison with experiment}.
\newblock {\em \textnormal{Physica A: Statistical Mechanics and its Applications}}, 637:129599, 2024.

\bibitem{einstein1905}
A.~Einstein.
\newblock \textit{{\"U}ber die von der molekularkinetischen theorie der w{\"a}rme geforderte bewegung von in ruhenden fl{\"u}ssigkeiten suspendierten teilchen}.
\newblock {\em \textnormal{Annalen der Physik}}, 17:549--560, 1905.

\bibitem{engel2000one}
K.-J. Engel and R.~Nagel.
\newblock {\em \textit{One-Parameter Semigroups for Linear Evolution Equations}}.
\newblock Springer-Verlag, New York, 2000.

\bibitem{cescudero}
C.~Escudero.
\newblock \textit{Kinetic energy of the Langevin particle}.
\newblock {\em \textnormal{Studies in Applied Mathematics}}, 145:719--738, 2020.

\bibitem{cescudero2}
C.~Escudero.
\newblock \textit{Fluctuation-dissipation relation, Maxwell-Boltzmann statistics, equipartition theorem, and stochastic calculus}.
\newblock {\em \textnormal{Physica Scripta}}, 98:055214, 2023.

\bibitem{em}
C.~Escudero and C.~Manada.
\newblock \textit{It{\^o} versus Stratonovich in a stochastic cosmological model}.
\newblock {\em \textnormal{Letters in Mathematical Physics}}, 112:12, 2022.

\bibitem{escudero2023ito}
C.~Escudero and H.~Rojas.
\newblock \textit{It{\^o} versus H{\"a}nggi-Klimontovich}.
\newblock arXiv preprint arXiv:2309.03654, 2023.

\bibitem{evanspde}
L.~C. Evans.
\newblock {\em \textit{Partial Differential Equations}}.
\newblock American Mathematical Society, Providence, 1998.

\bibitem{evans}
L.~C. Evans.
\newblock {\em \textit{An Introduction to Stochastic Differential Equations}}.
\newblock American Mathematical Society, Providence, 2013.

\bibitem{falconer}
K.~J. Falconer.
\newblock \textit{One-sided multifractal analysis and points of non-differentiability of devil's staircases}.
\newblock {\em \textnormal{Mathematical Proceedings of the Cambridge Philosophical Society}}, 136:167--174, 2004.

\bibitem{fehlberg}
E.~Fehlberg.
\newblock \textit{Low-order classical Runge-Kutta formulas with stepsize control and their application to some heat transfer problems}.
\newblock Technical Report R-315, NASA, 1969.

\bibitem{sanchez2012singular}
J.~Fern{\'a}ndez~S{\'a}nchez, P.~Viader, J.~Parad{\'\i}s, and M.~D{\'\i}az~Carrillo.
\newblock \textit{A singular function with a non-zero finite derivative}.
\newblock {\em \textnormal{Nonlinear Analysis: Theory, Methods \& Applications}}, 75:5010--5014, 2012.

\bibitem{fick1855}
A.~Fick.
\newblock \textit{Ueber diffusion}.
\newblock {\em \textnormal{Annalen der Physik}}, 94:59--86, 1855.

\bibitem{fixman1978}
M.~Fixman.
\newblock \textit{Simulation of polymer dynamics. I. General theory}.
\newblock {\em \textnormal{The Journal of Chemical Physics}}, 69:1527--1537, 1978.

\bibitem{golub2013matrix}
G.~H. Golub and C.~F. Van~Loan.
\newblock {\em \textit{Matrix Computations}}.
\newblock Johns Hopkins University Press, Baltimore, 2013.

\bibitem{hairer}
E.~Hairer, S.~P. N{\o}rsett, and G.~Wanner.
\newblock {\em \textit{Solving Ordinary Differential Equations I: Nonstiff Problems}}.
\newblock Springer-Verlag, Berlin, 1993.

\bibitem{henri2022rwpt}
C.~V. Henri and E.~Diamantopoulos.
\newblock \textit{Unsaturated transport modeling: Random-walk particle-tracking as a numerical-dispersion free and efficient alternative to Eulerian methods}.
\newblock {\em \textnormal{J.~Adv.~Model.~Earth Syst.}}, 14:e2021MS002812, 2022.

\bibitem{horn1994topics}
Roger~A. Horn and Charles~R. Johnson.
\newblock {\em \textit{Topics in Matrix Analysis}}.
\newblock Cambridge University Press, Cambridge, 1994.

\bibitem{hutter1998}
M.~H{\"u}tter and H.~C. {\"O}ttinger.
\newblock \textit{Fluctuation-dissipation theorem, kinetic stochastic integral and efficient simulations}.
\newblock {\em \textnormal{Journal of the Chemical Society, Faraday Transactions}}, 94:1403--1406, 1998.

\bibitem{ito1}
K.~It{\^o}.
\newblock \textit{Stochastic integral}.
\newblock {\em \textnormal{Proceedings of the Imperial Academy, Tokyo}}, 20:519--524, 1944.

\bibitem{ito2}
K.~It{\^o}.
\newblock \textit{On a stochastic integral equation}.
\newblock {\em \textnormal{Proceedings of the Imperial Academy, Tokyo}}, 22:32--35, 1946.

\bibitem{jeon2014scaled}
J.-H. Jeon, A.~V. Chechkin, and R.~Metzler.
\newblock \textit{Scaled Brownian motion: a paradoxical process with a time dependent diffusivity for the description of anomalous diffusion}.
\newblock {\em \textnormal{Phys. Chem. Chem. Phys.}}, 16:15811--15817, 2014.

\bibitem{kawamura2011set}
K.~Kawamura.
\newblock \textit{On the set of points where Lebesgue's singular function has the derivative zero}.
\newblock {\em \textnormal{Proc. Japan Acad., Ser. A, Math. Sci.}}, 87:162--166, 2011.

\bibitem{kubo1966}
R.~Kubo.
\newblock \textit{The fluctuation-dissipation theorem}.
\newblock {\em \textnormal{Reports on Progress in Physics}}, 29:255--284, 1966.

\bibitem{kuo}
H.~H. Kuo.
\newblock {\em \textit{Introduction to Stochastic Integration}}.
\newblock Springer, New York, 2006.

\bibitem{mp2010bm}
P.~M{\"o}rters and Y.~Peres.
\newblock {\em \textit{Brownian Motion}}.
\newblock Cambridge University Press, Cambridge, 2010.

\bibitem{nordam2023gmd}
T.~Nordam, R.~Kristiansen, R.~Nepstad, E.~van Sebille, and A.~M. Booth.
\newblock \textit{A comparison of Eulerian and Lagrangian methods for vertical particle transport in the water column}.
\newblock {\em \textnormal{Geosci.~Model Dev.}}, 16:5339--5363, 2023.

\bibitem{oksendal}
B.~{\O}ksendal.
\newblock {\em \textit{Stochastic Differential Equations: An Introduction with Applications}}.
\newblock Springer, Berlin, 2003.

\bibitem{pachecopozo2025njp}
A.~Pacheco-Pozo, I.~M. Sokolov, R.~Metzler, and D.~Krapf.
\newblock \textit{Heterogeneous diffusion in a harmonic potential: the role of the interpretation}.
\newblock {\em \textnormal{New J. Phys.}}, 27:064602, 2025.

\bibitem{paradis2001derivative}
J.~Parad{\'\i}s, P.~Viader, and L.~Bibiloni.
\newblock \textit{The derivative of Minkowski's $? (x)$ function}.
\newblock {\em \textnormal{Journal of Mathematical Analysis and Applications}}, 253:107--125, 2001.

\bibitem{perez2019reactive}
L.~Perez, J.~J. Hidalgo, and M.~Dentz.
\newblock \textit{Reactive random walk particle tracking and its equivalence with the advection-diffusion-reaction equation}.
\newblock {\em \textnormal{Water Resources Research}}, 55:847--855, 2019.

\bibitem{perezillanes2024modpathrw}
G.~P{\'e}rez-Illanes and D.~Fern{\'a}ndez-Garc{\'\i}a.
\newblock \textit{MODPATH-RW: A random walk particle tracking code for solute transport in heterogeneous aquifers}.
\newblock {\em \textnormal{Groundwater}}, 62:617--634, 2024.

\bibitem{protter2005stochastic}
P.~E. Protter.
\newblock {\em \textit{Stochastic Integration and Differential Equations}}.
\newblock Springer-Verlag, Berlin, 2005.

\bibitem{roberts2024cloaking}
C.~Roberts, Z.~Zhang, H.~Rojas, S.~Bo, C.~Escudero, S.~Guenneau, and G.~Pruessner.
\newblock \textit{Stochastic cloaking: concealing a region from diffusive particles}.
\newblock {\em \textnormal{Physical Review Research}}, 7:L032025, 2025.

\bibitem{spivakovskaya2007lagrangian}
D.~Spivakovskaya, J.~M. Redondo, and J.~Chernich.
\newblock \textit{Lagrangian modelling of multi-dimensional advection-diffusion with space-varying diffusivities: theory and idealized test cases}.
\newblock {\em \textnormal{Ocean Dynamics}}, 57:189--203, 2007.

\bibitem{stratonovich}
R.~L. Stratonovich.
\newblock \textit{A new representation for stochastic integrals and equations}.
\newblock {\em \textnormal{SIAM Journal on Control}}, 4:362--371, 1966.

\bibitem{stroock1979}
D.~W. Stroock and S.~R.~S. Varadhan.
\newblock {\em \textit{Multidimensional Diffusion Processes}}.
\newblock Springer-Verlag, New York, 1979.

\bibitem{wy}
S.~Watanabe and T.~Yamada.
\newblock \textit{On the uniqueness of solutions of stochastic differential equations}.
\newblock {\em \textnormal{Journal of Mathematics of Kyoto University}}, 11:155--167, 1971.

\bibitem{zwanzig2001}
R.~Zwanzig.
\newblock {\em \textit{Nonequilibrium Statistical Mechanics}}.
\newblock Oxford University Press, New York, 2001.

\end{thebibliography}

\vskip10mm
\noindent
{\footnotesize
Carlos Escudero\par\noindent
Departamento de Matem\'aticas Fundamentales\par\noindent
Universidad Nacional de Educaci\'on a Distancia\par\noindent
{\tt cescudero@mat.uned.es}\par\vskip1mm\noindent
}
\vskip2mm
\noindent
{\footnotesize
Helder Rojas\par\noindent
Departamento de Matem\'aticas Fundamentales\par\noindent
Universidad Nacional de Educaci\'on a Distancia\par\noindent
{\tt hrojas36@alumno.uned.es}\par\vskip1mm\noindent
}
\end{document}